\documentclass[12pt,letterpaper]{article}
\usepackage{amsmath,amsthm,amsfonts,amssymb,amscd}
\usepackage{authblk}
\usepackage[shortlabels]{enumitem}
\usepackage{fancyhdr}
\usepackage[margin=2cm,bottom=4cm,marginparwidth=2cm]{geometry}
\usepackage{graphicx}
\usepackage{lastpage}
\usepackage[outline]{contour}
\usepackage[numbers,sort]{natbib}
\usepackage{thmtools}
\usepackage[dvipsnames]{xcolor}
\usepackage{bbm}

\usepackage[pdfencoding=auto]{hyperref}
\usepackage[capitalise,nameinlink]{cleveref}
\hypersetup{
  colorlinks   = true, 
  urlcolor     = blue, 
  linkcolor    = blue, 
  citecolor    = black 
}

\theoremstyle{definition}
\newtheorem{definition}{Definition}[section]
\newtheorem{example}[definition]{Example}

\theoremstyle{plain}
\newtheorem{theorem}[definition]{Theorem}
\newtheorem{lemma}[definition]{Lemma}
\newtheorem{proposition}[definition]{Proposition}
\newtheorem{corollary}[definition]{Corollary}
\newtheorem{remark}[definition]{Remark}

\def\CC{{\mathbb C}}

\def\NN{{\mathbb N}}

\def\PP{{\mathbb P}}

\def\RR{{\mathbb R}}

\def\ZZ{{\mathbb Z}}

\renewcommand{\phi}{\varphi}

\def\ones{\mathbbm{1}}

\newcommand{\inv}{^{-1}}
\newcommand{\dl}{\partial}

\newcommand{\til}{\widetilde}

\def\conv{\operatorname {conv}}

\DeclareMathOperator{\Int}{Int}
\DeclareMathOperator{\rank}{rank}
\DeclareMathOperator{\sign}{sign}
\def\supp{\operatorname {supp}}

\DeclareMathOperator{\diag}{diag}

\def\spanset{\operatorname {span}}

\def\lbbrace{\{\!\{}
\def\rbbrace{\}\!\}}

\renewcommand\vec{\mathbf}

\title{Lorentzian Symmetric Polynomials}
\author[1]{Tracy Chin}
\author[2]{Daniel Qin}
\affil[1]{University of Washington}
\affil[2]{University of California, Davis}
\date{\today}

\begin{document}

\maketitle
\begin{abstract}
We study the class of Lorentzian symmetric polynomials and Lorentzian symmetric functions, which are defined to be symmetric functions for which every truncation of variables is Lorentzian. Similar to the space of Lorentzian polynomials, we show that the space of Lorentzian symmetric polynomials is homeomorphic to a closed Euclidean ball. Our main result is a reduction scheme that significantly reduces the complexity of testing for Lorentzianity. Using this method, we provide explicit semialgebraic descriptions of the spaces of Lorentzian symmetric polynomials and functions for degrees up to six. These techniques can also be applied to simplify the proofs to known cases of Lorentzian symmetric functions. We conclude by showing that some natural symmetric operators fail to preserve Lorentzianity which in turn highlights an inherent tension between symmetry in variables and the Lorentzian property.
\end{abstract}
\section{Introduction}
Lorentzian polynomials have been an extremely active area of research over the past few years. They have been a particularly useful tool in proving conjectures about (ultra) log-concavity of sequences, as they provide a link between the continuous log-concavity of a polynomial as a function on $\RR^n_+$ and the discrete log-concavity of its coefficients.

The emergence of Lorentzian polynomials traces back to the independent works of \citeauthor*{anari2024masons3} \cite{anari2024masons3} and \citeauthor*{branden2019hodgeriemannrelationspottsmodel} \cite{branden2019hodgeriemannrelationspottsmodel}, who simultaneously used the theory to prove Mason's ultra log-concavity conjecture on the number of independent sets of a matroid. Since then, Lorentzian polynomials have been used in obtaining results towards long-standing log-concavity conjectures. For example, in \cite{hafner2025alexanderpolynomialspecialalternating}, \citeauthor*{hafner2025alexanderpolynomialspecialalternating} used Lorentzian polynomials to prove a special case of Fox's conjecture, and in \cite{alexanderssonjal2025rookmatroids}, \citeauthor{alexanderssonjal2025rookmatroids} used them to prove the log-concavity consequence of the Neggers-Stanley conjecture for naturally labeled width two posets.

Additionally, the support of a Lorentzian polynomial is always an $M$-convex set, meaning that the monomials appearing in a Lorentzian polynomial are always the integer points of some generalized permutahedron. This fact was used, for example, to show that the mixed volume polynomial of any collection of convex bodies has $M$-convex support \cite[Corollary 4.2]{BrandenHuh20Lorentzian}.

In this paper, we explore the intersection of Lorentzian and symmetric polynomials, with an eye towards establishing useful techniques and general theory. One major paper in this intersection is \cite{huh2022schurlorentzian}, in which \citeauthor*{huh2022schurlorentzian} show that Schur polynomials are denormalized Lorentzian, and, as a corollary, show that Kostka numbers are log concave in root directions. Their proof relies heavily on tools from algebraic geometry rather than combinatorics. A second noteworthy result in the same paper is that weight multiciplicities of shifted Verma modules over $\mathfrak{sl}_{n+1}(\CC)$ are log-concave. More recently \citeauthor*{khare2025logconcavitycharactersparabolicverma} extended the latter result to the class of parabolic Verma modules \cite{khare2025logconcavitycharactersparabolicverma}.

A large part of this paper is devoted to showing that symmetric polynomials are particularly amenable to using combinatorial techniques to test for Lorentzianity. While testing Lorentzianity for a fixed degree is always polynomial in the number of variables \cite{chin2024realstabilitylogconcavity}, for symmetric polynomials, the number of arithmetic operations to check the Lorentzian property is independent of the number of variables (see \cref{thm:const-time}). In \cref{sec:general-theory}, we give general results on the space of Lorentzian symmetric polynomials. In \cref{subsec:euclidean-ball}, we study the topology of the space of Lorentzian symmetric polynomials, and specifically show that it is homeomorphic to a closed Euclidean ball (\cref{thm:proj sym lor=ball}). Then, in \cref{subsec:reductions}, we give structural results on Lorentzian symmetric polynomials, including describing the structures that enable constant time algorithms for checking Lorentzianity. In \cref{sec:small degree}, we then use these structures to give explicit semialgebraic descriptions of the space of Lorentzian symmetric polynomials for small degrees.

In \cref{sec:simplifying-results}, we apply our techniques to simplify the proofs of two known results on Lorentzian symmetric polynomials. First, we tackle the result of \citeauthor*{Matherne24chromaticsymmetric} showing that chromatic symmetric polynomials of certain graphs are Lorentzian; in \cref{subsec:chromatic-symmetric}, we use our techniques to simplify the computations in their proof considerably. We also improve on the work of the second author in \cite{qinthesis}, in which he uses combinatorial techniques to show that Schur polynomials of two-column tableaux are denormalized Lorentzian. In \cref{subsec:two-column-schur}, we use our techniques to give a simpler proof, again combinatorial in nature, of this special case.

Finally, in \cref{sec:non-results}, we explore the ways in which symmetric polynomials and the Lorentzian property may be incompatible. Specifically, we show that two natural operations on symmetric polynomials, the $\omega$-involution and Hall inner product, do not preserve Lorentzianity.

\paragraph{Acknowledgements.} We would like to thank Aryaman Jal and Cynthia Vinzant for their helpful comments. Tracy Chin was partially supported by NSF grant DMS-2153746. Daniel Qin was partially supported by NSF grant CCF-2317280. Work on this project began at the eighth edition of Encuentro Colombiano de Combinatoria (ECCO) in Popay{\'a}n, Colombia. Daniel Qin would also like to thank Petter Brändén for introducing him to Lorentzian polynomials, and Melissa Zhang for financial support to attend ECCO.

\section{Background and Notation}~\label{sec:background}
We begin by fixing some notation and recalling the definitions of the objects under study. We write $\RR_{\geq 0}[x_1,\dots,x_n]_d$ to denote the set of all homogeneous polynomials of degree $d$ on $n$ variables with nonnegative coefficients. Throughout this paper, we use the shorthand $\dl_i f := \frac{\dl}{\dl x_i} f$ and similarly $\dl_i^k f := \frac{\dl^k}{\dl x_i^k}f$. Additionally, for $\alpha = (\alpha_1, \dots, \alpha_n) \in \ZZ_{\geq 0}^n$, we use the shorthand $\dl^\alpha f := \dl_1 ^{\alpha_1} \dl_2^{\alpha_2} \dots \dl_n^{\alpha_n} f$, and we define $|\alpha| := \sum_{i=1}^n \alpha_i$.

\subsection{Lorentzian Polynomials}
Lorentzian polynomials were initially introduced simultaneously and independently in \cite{BrandenHuh20Lorentzian} and \cite{anari2024masons3} under different, but equivalent, definitions. In \cite{anari2024masons3}, they are called \emph{completely log-concave polynomials} and are defined in terms of a log-concavity condition on directional derivatives.

In \cite[Theorem 2.30]{BrandenHuh20Lorentzian}, \citeauthor*{BrandenHuh20Lorentzian} show that this is equivalent to their definition of Lorentzian polynomials. By \cite[Theorem 3.2]{anari2024masons3}, both definitions are also equivalent to \cref{def:lorentzian}, which is the definition we use in this paper.

\begin{definition}\label{def:lorentzian}
    Let $f\in \RR_{\geq 0}[x_1,\ldots,x_n]_d$ be a homogeneous polynomial of degree $d$. We say $f$ is \emph{Lorentzian} if both of the following hold.
    \begin{enumerate}[(i)]
        \item[(S)] For all $\alpha \in \ZZ_{\geq 0}^n$, $|\alpha| \leq d-2$, $\dl^\alpha f$ is indecomposable (i.e. it cannot be written as the sum of two nonzero polynomials on disjoint sets of variable).
        \item[(H)] For all $\alpha \in \ZZ_{\geq 0}^n$, $|\alpha| = d-2$, $\nabla^2 \dl^\alpha f$ has at most one positive eigenvalue.
    \end{enumerate}
    We write $L^d_n$ to denote the set of all Lorentzian polynomials of degree $d$ on $n$ variables.
\end{definition}

We also work with denormalized Lorentzian polynomials. The normalization operator takes generating functions to exponential generating functions. The interactions between this operation and the Lorentzian property was first explored in \cite{BrandenHuh20Lorentzian}, and further expanded in \cite{huh2022schurlorentzian,branden2023lowerbounds}.
\begin{definition}
    Given a polynomial $f = \sum_{\alpha \in \ZZ_{\geq 0}^n} c_\alpha x^\alpha$, we define its \emph{normalization} as
    \[N(f) := \sum_{\alpha \in \ZZ_{\geq 0}^n} c_\alpha \frac{x^{\alpha}}{\alpha!},\]
    where $\alpha! := \prod_{i=1}^n \alpha_i!$. We say $f$ is \emph{denormalized Lorentzian} if $N(f)$ is Lorentzian.
\end{definition}
We will also find the following result useful for analyzing the signatures of matrices, though it is purely linear algebra and already known as folklore. For example, a similar statement for strictly positive matrices appears in \cite[Theorem 4.4.6]{bapat1997nonnegativematrices}. However, for the sake of completeness, we reproduce it here. 
\begin{proposition}~\label{prop:signature-from-minors}
    Let $A \in \RR^{n\times n}_{\geq 0}$ be a symmetric matrix with nonnegative entries. Then $A$ has at most one positive eigenvalue if and only if for all nonempty subsets $S \subseteq [n]$, $(-1)^{|S| - 1}A_S \geq 0$, where $A_S = \det(A[S])$ denotes the principal minor with rows and columns indexed by $S$.
\end{proposition}
\begin{proof}
    \textbf{(\contour{black}{$\Rightarrow$})} Let $A \in \RR^{n\times n}_{\geq 0}$ be a symmetric matrix and assume that $A$ has at most one positive eigenvalue. By Cauchy interlacing, for any $S \subseteq [n]$, $A[S]$ also has at most one positive eigenvalue. If $A[S]$ is not invertible, then $(-1)^{|S|-1}A_S = 0$, and the desired inequality holds. Otherwise, we have $A_S \neq 0$, and in particular $A[S]$ is not the zero matrix. Since $A$ has nonnegative entries, this implies that $\ones^T A[S] \ones > 0$, so $A[S]$ has at least one positive eigenvalue. Therefore, $A[S]$ has exactly one positive eigenvalue, and since we assumed it was invertible, it must have exactly $|S|-1$ negative eigenvalues. Thus, $\sign(A_S) = (-1)^{|S|-1}$, so $(-1)^{|S|-1}A_S \geq 0$, as desired.

    \textbf{(\contour{black}{$\Leftarrow$})} We proceed by induction. As a base case, let $n = 2$. By assumption, we have $\det(A) \leq 0$. If $\det(A) < 0$, then $A$ must have exactly one positive and one negative eigenvalue, as desired. If $\det(A) = 0$, then $\lambda = 0$ must be an eigenvalue of $A$. Since $A$ only has two eigenvalues, then it can have at most one positive eigenvalue, as desired.

    Now, suppose the result holds for $k \times k$ matrices for all $2 \leq k < n$, and let $A$ be an $n\times n$ symmetric matrix satisfying the assumptions. By induction, we know that all $(n-1)\times(n-1)$ principal submatrices of $A$ have at most one positive eigenvalue, so by Cauchy interlacing, $A$ has at most two positive eigenvalues. Suppose for the sake of contradiction that $A$ has two positive eigenvalues. This implies that $(-1)^{n-2}\det(A) \geq 0$, but by assumption we have $(-1)^{n-1}\det(A) \geq 0$, so we must have $\det(A) = 0$. Therefore $\rank(A) = r < n$, and, writing $\lambda_1 \geq \dots \geq \lambda_n$ for the eigenvalues of $A$, we have
    \[\lambda_3(A) = \dots = \lambda_{n-r+2}(A) = 0.\]

    Since $A$ is symmetric, $\rank(A)$ is equal to the maximum size of an invertible principal submatrix of $A$. Let $B = A[S]$ be an invertible $r\times r$ principal submatrix of $A$. By the inductive hypothesis, we know that $\lambda_2(B) \leq 0$. On the other hand, by Cauchy interlacing, we have
    \[
        \lambda_2(B) \geq \lambda_{n-r+2}(A) = 0.
    \]
    However, this implies that $\lambda_2(B) = 0$, which contradicts the invertibility of $B$. Hence, $A$ has at most one positive eigenvalue, as desired.
\end{proof}
    
\subsection{Symmetric Polynomials}
A symmetric function over a ring $R$ is an element $f \in R[[x]]$ such that for any permutation $\sigma: \NN \to \NN$, $f(x_1,x_2,\dots) = f(x_{\sigma(1)}, x_{\sigma(2)},\dots)$. We write $\Lambda_\RR$ (or simply $\Lambda$) for the ring of symmetric functions over $\RR$. The ring of symmetric functions naturally admits a grading by degree; we write $\Lambda^d$ for the degree $d$ graded piece.

Similarly, a symmetric polynomial is an element $f \in R[x_1,\dots,x_n]$ such that for any $\sigma \in S_n$, $f(x_1,\dots,x_n) = f(x_{\sigma(1)},\dots, x_{\sigma(n)})$. We write $\Lambda_n$ to denote the ring of symmetric polynomials over $\RR$ in $n$ variables. We note that $\Lambda_n$ is again graded by degree, and we write $\Lambda_n^d$ for the degree $d$ graded piece.

A simple basis for $\Lambda$ as an $\RR$-vector space is the set of monomial symmetric functions
\[m_\lambda = \sum_{\alpha} x_1^{\alpha_1} x_2^{\alpha_2}\dots,\]
where $\alpha$ ranges over all permutations of the vector $\lambda = (\lambda_1, \lambda_2, \dots)$. In this paper, we primarily work with the normalized monomial basis
\[\til{m}_\lambda := N(m_\lambda) = \frac{m_\lambda}{\lambda!}\] because the normalization makes studying derivatives much cleaner.

We will also work with another basis for symmetric functions, namely Schur symmetric functions. These are again indexed by partitions and defined by
\[s_\lambda = \sum_\mu K_{\lambda,\mu} m_\mu,\]
where the Kostka number $K_{\lambda,\mu}$ is the number of semistandard Young tableaux (SSYT) of shape $\lambda$ and content $\mu$.

\subsection{Partitions}
We will also require some background on integer partitions. Recall that a partition $\lambda = (\lambda_1, \lambda_2, \dots)$ of a nonnegative integer $n$ is a non-increasing sequence of nonnegative numbers $\lambda_1 \geq \lambda_2 \geq \dots$ such that $\sum \lambda_i = n$. The length of $\lambda$, denoted $\ell(\lambda)$, is the number of nonzero entries in $\lambda$.

We recall the dominance order on partitions, which puts a lattice structure on the set of partitions of a given integer $n$.
\begin{definition}
    \emph{Dominance order} is the partial order on partitions of the same size defined as follows. If $\lambda, \mu \vdash n$, we say $\lambda \preceq \mu$ if $\sum_{i=1}^k \lambda_i \leq \sum_{i=1}^k \mu_i$ for all $k$.
\end{definition}

We also work with certain structural properties of partitions. Let $\mu = (\mu_1, \mu_2, \dots, \mu_k)$ be a partition of length $\ell(\mu) = k$. Suppose $\mu$ has the structure  
\[
    \mu = (\mu_1=\ldots=\mu_{m_2-1}>\mu_{m_2}=\ldots=\mu_{m_3-1}>\ldots>\mu_{m_{\ell}}=\ldots=\mu_{m_{\ell+1}-1},0,\ldots,0).
\]
This gives us a set partition $(\beta_1,\dots,\beta_\ell)$ of $[k]$ with $\beta_t = [m_t, m_{t+1})$, where $m_1 = 1$ and $m_{\ell + 1} = k+1$. We define $n_t := m_{t+1} - m_t$ to be the size of block $\beta_t$.

By construction, $\mu_i = \mu_j$ if and only if $i$ and $j$ are in the same part of this set partition. Hence, $n_t$ is equivalently the number of times the $t$th largest distinct entry appears in $\mu$. We note that $\mu$ also induces a set partition of $[n]$ by adding one more block $\beta_{\ell + 1} = [k+1, n]$ to this partition. This new partition again has the property that $\mu_i = \mu_j$ if and only if $i$ and $j$ are in the same block.

Every partition $\lambda$ is also associated to a natural polytope called a permutohedron.
\begin{definition}~\label{def:permutohedron}
Let $\lambda=(\lambda_1\geq \ldots\geq \lambda_n)$ be a partition of $d.$ The \emph{permutohedron} $P(\lambda)\subseteq \{x_1+\cdots+x_n=d\}\subseteq \RR^n$ is given by $\conv\{\lambda^\sigma:\sigma\in S_n\}$ where $\lambda^\sigma$ denotes the composition $\lambda^\sigma = (\lambda_{\sigma(1)},\dots,\lambda_{\sigma(n)})$ for all $\sigma\in S_n.$
\end{definition}
A classical result of Rado \cite{RR52} characterizes integer points of permutohedra. 
\begin{proposition}~\label{prop:rado}
    Let $\lambda=(\lambda_1\geq \ldots\geq \lambda_n)$ be a partition of $d.$ A point $t=(t_1,\ldots,t_n)\in\RR^n$ belongs to $P(\lambda_1,\ldots,\lambda_n)$ if and only if $t$ is a weak composition of $d$
    and for any nonempty subset $\{i_1,\ldots,i_k\}\subseteq[n]$ we have $$t_{i_1}+\ldots+t_{i_k}\leq \lambda_1+\ldots+\lambda_k.$$
\end{proposition}
In particular, the second condition simplifies to dominance when $t$ is a partition. A short proof of the simplification can be found in~\cite{qinthesis}. Later on, this will give us an immediate alternative interpretation of supports of symmetric M-convex sets.
\begin{proposition}~\label{prop:radorev}
A partition $\mu\vdash d$ belongs to $P(\lambda)$ if and only if $\mu\preceq \lambda.$
\end{proposition}

\section{General Theory on Lorentzian Symmetric Polynomials}\label{sec:general-theory}

In this section, we establish some of the general properties of symmetric functions that make them particularly amenable to testing the Lorentzian property. First, however, we need to address the subtlety that a symmetric function is not a polynomial, since they are traditionally defined on an infinite set of variables. However, by restricting to finite numbers of variables, we can think of a symmetric function as defining a family of symmetric polynomials, and we define a symmetric function to be Lorentzian if all polynomials in this family are Lorentzian.

\begin{definition}\label{def:lorentzian-symmetric-function}
    We say a symmetric function $f\in\Lambda^d$ is Lorentzian if for all $n$, the symmetric polynomial $f_n(x_1,\dots,x_n):=f(x_1,\dots,x_n,0,\dots) \in \Lambda^d_n$ is Lorentzian.
\end{definition}

For example, normalized Schur polynomials are one family of Lorentzian symmetric functions \cite{huh2022schurlorentzian}, as are chromatic symmetric functions of certain graphs arising from abelian Dyck paths \cite{Matherne24chromaticsymmetric}.

The finite polynomial specializations in this definition are also naturally related to one another. In particular, if $f_n$ is Lorentzian, then so is $f_k$ for all $k \leq n$.

\begin{proposition}\label{prop:specialization-descending-chain}
    Let $\{m_\alpha(n)\}_{\alpha \vdash d} \subseteq \RR[x_1,\dots,x_n]$ denote the monomial basis for $\Lambda_n^d$, and let
    \[A_n = \left\{(c_\alpha)_{\alpha\vdash d} \mid f_n(x_1,\dots,x_n) :=\sum c_\alpha m_\alpha(n) \text{ is Lorentzian}\right\}.\]

    \noindent Then these sets form a descending chain $A_d \supseteq A_{d+1} \supseteq \dots$, and the intersection $\bigcap_{n=d}^\infty A_n$ is non-empty.
\end{proposition}
\begin{proof}
    First, we note that $m_\alpha(n-1) = \left.m_\alpha(n)\right\rvert_{x_n=0}$. Thus, $f_{n-1}(x_1,\dots,x_{n-1}) = f_n(x_1,\dots,x_{n-1},0)$. Since specialization preserves Lorentzianity, this implies that $A_n \subseteq A_{n-1}$, as desired.

    Moreover, the elementary symmetric polynomial $e_d(x_1,\dots,x_n)$ is Lorentzian for all $n$ \cite[Example 2.27]{BrandenHuh20Lorentzian}, so the intersection $\bigcap_{n=d}^\infty A_n$ is nonempty.
\end{proof}

Lorentzian symmetric functions inherit many useful properties from Lorentzian polynomials. For example, products of Lorentzian symmetric functions are again Lorentzian.

\begin{proposition}
    Suppose $f \in \Lambda^d$, $g \in \Lambda^e$ are Lorentzian symmetric functions. Then $fg \in \Lambda^{d+e}$ is again a Lorentzian symmetric function.
\end{proposition}
\begin{proof}
    We note that $(fg)_n = f_n g_n$, so it suffices to show that $f_ng_n$ is Lorentzian for all $n$. Indeed, $f_n$ and $g_n$ are Lorentzian by definition, and the product of Lorentzian polynomials is Lorentzian \cite[Corollary 2.32]{BrandenHuh20Lorentzian}, which completes the proof.
\end{proof}

We can also study the class of linear operators that preserve Lorentzianity of symmetric functions. For example, we consider the symmetric functions analog of taking partial derivatives.
\begin{example}
    Let $\frac{\dl}{\dl p_1}: \Lambda^d \to \Lambda^{d-1}$ be the derivative with respect to $p_1$, acting on symmetric functions expressed as polynomials in the power sum basis. We claim that this is a Lorentzianity-preserving linear operator on symmetric functions.

    Indeed, we claim that $\left(\frac{\dl}{\dl p_1} f\right)_{n} = \left.\frac{\dl}{\dl x_{n+1}}\right\rvert_{x_{n+1} = 0} f_{n+1}$. The right-hand side of this equation is Lorentzian whenever $f_{n+1}$ is Lorentzian, so proving this equality is sufficient to show that $\frac{\dl}{\dl p_1}$ is Lorentzianity-proving.

    By linearity, it suffices to show that $\left(\frac{\dl}{\dl p_1}p_\lambda\right)_n = \left.\frac{\dl}{\dl x_{n+1}}\right\rvert_{x_{n+1} = 0} p_\lambda(x_1,\dots,x_{n+1})$. Let $\lambda = (1^{m_1}2^{m_2}\dots)$ be the partition with $m_1$ 
    ones, $m_2$ twos, and so on. Then we have
    \[\left(\frac{\dl}{\dl p_1}p_\lambda\right)_n = m_1\cdot p_{(1^{m_1-1}2^{m_2}\dots)}(x_1,\dots,x_n).\]
    On the other hand, we note that
    \[\left.\frac{\dl}{\dl x_{n+1}}\right\rvert_{x_{n+1}=0} p_r = \begin{cases}
        1, & \text{if } r = 1\\
        0, & \text{otherwise}
    \end{cases}.\]
    Therefore,
    \[\left.\frac{\dl}{\dl x_{n+1}}\right\rvert_{x_{n+1}=0} p_\lambda(x_1,\dots,x_{n+1}) = m_1 p_{(1^{m_1-1}2^{m_2}\dots)},\]
    which completes the proof.
\end{example}

We also get some version of a symbol theory for linear operators preserving Lorentzian symmetric functions, though the statement is not nearly as neat as it is for the polynomial case.
\begin{proposition}
    Let $T: \Lambda^d \to \Lambda^\ell$ be a linear map on symmetric functions. For $n \geq \ell$, define $T_n: \Lambda^d_n \to \Lambda^\ell_n$ by $T_n(m_\lambda(x_1,\dots,x_n)) = \left.T(m_\lambda)\right\rvert_{x_{n+1} = x_{n+2} = \dots = 0}$. Suppose that for all $n \geq \ell$,
    \[\operatorname{Symb}_{T,n}(w,z) := \sum_{\substack{\lambda \text{ partition}\\ \ell(\lambda) \leq n, \lambda_1 \leq d}}\binom{\mathbf{d}}{\lambda} \frac{T_n(m_\lambda(w_1,\dots,w_n))}{m_\lambda(1^n)}m_{\mathbf{d} - \lambda}(z_1,\dots,z_n)\]
    is Lorentzian, where we use $\mathbf{d}$ as shorthand for $(d,\dots,d)\in\ZZ^n_{\geq 0}$. Then $T$ is a Lorentzianty-preserving operator on symmetric functions.
\end{proposition}
\begin{proof}
    By construction, $(T(f))_n = T_n(f_n)$, so it suffices to show that $T_n: \Lambda^d_n \to \Lambda^\ell_n$ is Lorentzianity-preserving for all sufficiently large $n$. We note that we can extend $T_n$ to a linear map $\til{T}_n:\RR[x_1,\dots,x_n] \to \RR[x_1,\dots,x_n]$ by defining $\til{T}_n(x^\alpha) := \frac{1}{m_{[\alpha]}(1^n)}T_n(m_{[\alpha]}(x_1,\dots,x_n))$, and that it suffices to show that $\til{T}_n$ preserves Lorentzianity.

    By \cite[Theorem 3.2]{BrandenHuh20Lorentzian}, it suffices to show that
    \[\operatorname{Symb}_{\til{T}_n} = \sum_{0 \leq \alpha \leq \mathbf{d}}\binom{\mathbf{d}}{\alpha}\til{T}_n(w^\alpha)z^{\mathbf{d}-\alpha}\]
    is Lorentzian.

    By construction, $\til{T}_n(w^\alpha)$ depends only on the partition $[\alpha]$ obtained by reordering the parts of $\alpha$, so
    \begin{align*}
        \operatorname{Symb}_{\til{T}_n} &= \sum_{\substack{\lambda \text{ partition}:\\
        \ell(\lambda) \leq n, \lambda_1 \leq d}} \binom{\mathbf{d}}{\lambda} \sum_{\alpha \til \lambda} \til{T}_n(w^\alpha)z^{\mathbf{d}-\alpha}\\
        &=\sum_{\substack{\lambda \text{ partition}:\\
        \ell(\lambda) \leq n, \lambda_1 \leq d}} \binom{\mathbf{d}}{\lambda} \sum_{\alpha \sim \lambda} \til{T}_n(w^\alpha)z^{\mathbf{d}-\alpha}\\
        &=\sum_{\substack{\lambda \text{ partition}:\\
        \ell(\lambda) \leq n, \lambda_1 \leq d}} \binom{\mathbf{d}}{\lambda} \frac{T_n(m_\lambda(w_1,\dots,w_n))}{m_\lambda(1^n)}\sum_{\alpha \sim \lambda}z^{\mathbf{d}-\alpha}\\
        &=\sum_{\substack{\lambda \text{ partition}:\\
        \ell(\lambda) \leq n, \lambda_1 \leq d}} \binom{\mathbf{d}}{\lambda} \frac{T_n(m_\lambda(w_1,\dots,w_n))}{m_\lambda(1^n)}m_{\mathbf{d} - \lambda}(z_1,\dots,z_n),
    \end{align*}
    which is exactly $\operatorname{Symb}_{T,n}(w,z)$. Hence, if $\operatorname{Symb}_{T,n}$ is Lorentzian for all $n$, then $T_n$ is Lorentzianity-preserving for all $n$, and so $T$ preserves Lorentzianity of symmetric functions, as desired.
\end{proof}

While we can naturally extend any symmetric polynomial to a symmetric function by simply taking the same coefficients in the monomial basis, this extension need not preserve the Lorentzian property. The semialgebraic descriptions we find in \cref{sec:small degree} include the number of variables as an input; by taking $n \to \infty$ in these inequalities, one recovers a semialgebraic description for Lorentzian symmetric functions.

\subsection{Symmetric M-Convex Sets}
    One necessary condition for a polynomial to be Lorentzian is that its support must be M-convex, in the sense of \cite{murota2003discreteconvexanalysis}. In fact, any M-convex set is the support of some Lorentzian polynomial \cite[Theorem 3.10]{BrandenHuh20Lorentzian}. Recall that we say a set $J \subseteq \ZZ^n$ is \emph{M-convex} if for all $x,y \in J$ and all $i \in [n]$ such that $x_i > y_i$, there exists an index $j$ satisfying $x_j < y_j$ such that $x-e_i+e_j$ and $y+e_i-e_j$ are both in $J$.
    
    If $f$ is also symmetric, this puts additional constraints on its support. In this section, we fully characterize the possible supports of Lorentzian symmetric functions and polynomials.
    
    \begin{definition}
        Let $f = \sum_{k=0}^d \sum_{\lambda \vdash k} c_\lambda m_\lambda$. We define the \emph{$m$-support} of $f$ to be $\supp_m(f) = \{\lambda: c_\lambda \neq 0\}$.
    \end{definition}
    
    Specifically, we show that if $f$ is a Lorentzian symmetric function or polynomial, then the $m$-support of $f$ is an order ideal of the dominance lattice with a unique maximal element.

    \begin{theorem}\label{thm:sym-lor-supp-sufficiency}
        If $f$ is a Lorentzian symmetric polynomial of degree $d$ on $n \geq d$ variables, then there is some $\lambda\vdash d$ such that $\supp_m(f) = \{\mu : \mu\preceq\lambda\}$. That is, the $m$-support of $f$ is an interval $[1^d, \lambda]$ in dominance order.
    \end{theorem}

    We prove this theorem in two steps. First we show that $\supp_m(f)$ has a unique maximal element. Then, we show that M-convexity implies that all partitions beneath this maximal element in dominance order must also be in the $m$-support of $f$.

    \begin{lemma}\label{lem:sym-lorentzian-unique-maxl}
        Suppose $f = \sum_{\lambda\vdash d} c_\lambda m_\lambda$ is a Lorentzian symmetric polynomial. Then $\supp_m(f)$ has a unique maximal element with respect to dominance order.
    \end{lemma}
    \begin{proof}
        Suppose for the sake of contradiction that there are two distinct maximal elements $\lambda$ and $\mu$. Since $\mu \not\succeq \lambda$, we have $\lambda_1 + \dots + \lambda_k > \mu_1 + \dots + \mu_k$ for some $k$. By picking the smallest such $k$, we have $\lambda_k > \mu_k$.
    
        Since $f$ is Lorentzian, $\supp(f)$ is M-convex, so there exists some $j$ such that $\mu_j > \lambda_j$ and $\lambda - e_k + e_j$ and $\mu + e_k - e_j$ are in $\supp(f)$. However, we claim that the maximality of $\lambda$ and $\mu$ means that no such $j$ can exist.
    
        First, if $j < k$, then we claim that $\lambda - e_k + e_j \succ \lambda$, so $\lambda$ is not maximal. Indeed, for any $r$ such that $1 \leq r < j$ or $k \leq r \leq \ell(\lambda)$, $\sum_{i=1}^r (\lambda - e_k + e_j)_i = \sum_{i=1}^r \lambda_i$. Moreover, for any $j \leq r < k$, $\sum_{i=1}^r (\lambda -e_k + e_j)_i = \sum_{i=1}^r\lambda_i + 1 > \sum_{i=1}^r\lambda_i$. Thus, $\lambda$ is not maximal. On the other hand, if $j > k$, then by the same logic, $\mu + e_k - e_j \succ \mu$, so $\mu$ is not maximal.
    
        Thus, no such $j$ can exist, so $f$ does not have M-convex support, which contradicts our assumption that $f$ is Lorentzian.
    \end{proof}

    To prove \cref{thm:sym-lor-supp-sufficiency}, it just remains to show that all partitions below the unique maximal element are also in the $m$-support of $f$.
    
    \begin{proof}[Proof of \cref{thm:sym-lor-supp-sufficiency}]
        By \cref{lem:sym-lorentzian-unique-maxl}, we know that if $f$ is Lorentzian, then $\supp_m(f)$ has a unique maximal element $\lambda$, so $\supp_m(f) \subseteq [1^d,\lambda]$. It remains to show that if $\mu \preceq\lambda$, then $\mu\in\supp_m(f)$.
    
        Suppose $f$ is a symmetric function of degree $d$ on $n$ variables, and let $\lambda = (\lambda_1 \geq \lambda_2 \geq \dots \geq \lambda_n) \vdash d$. As in \Cref{def:permutohedron}, for any $\sigma\in S_n$, let $\lambda^\sigma$ denote the composition $\lambda^\sigma = (\lambda_{\sigma(1)},\dots,\lambda_{\sigma(n)})$. Since $c_\lambda \neq 0$, $\lambda^\sigma \in \supp(f)$ for all $\sigma \in S_n$. Since $\supp(f)$ is M-convex, this implies that $\supp(f) \supseteq P(\lambda) \cap \ZZ^n$.

        Moreover, by \cref{prop:radorev}, if we have a partition $\mu\vdash d$ such that $\mu \preceq\lambda$ in dominance order, then $\mu \in P(\lambda) \cap \ZZ^n$. Hence $\mu \in \supp(f)$, and so $\mu \in \supp_m(f)$ whenever $\mu \preceq \lambda$, as desired.
    \end{proof}

    Following \cite[Theorem 3.10]{BrandenHuh20Lorentzian}, this gives us exactly the set of all $m$-supports of Lorentzian symmetric functions and polynomials.
    \begin{theorem}
        Let $J \subseteq \{\mu : \mu\vdash d\}$ be a set of partitions of $d$. The following are equivalent.
        \begin{enumerate}[(i)]
            \item There is a Lorentzian symmetric function whose $m$-support is $J$.
            \item There is a Lorentzian symmetric polynomial on $n \geq d$ variables whose $m$-support is $J$.
            \item $J$ is an interval $[1^d, \lambda]$ in dominance order.
            \item $J$ is the $m$-support of some Schur polynomial $s_\lambda.$ 
        \end{enumerate}
        Furthermore, if we define $S_n\cdot J := \bigcup_{\sigma\in S_n} \{\mu^\sigma : \sigma \in S_n\}$, then these are also equivalent to
        \begin{enumerate}
            \item[(v)] $S_n \cdot J$ is the set of integer points of a permutohedron $P(\lambda)$.
        \end{enumerate}
    \end{theorem}
    \begin{proof}
        We note that (i) $\Rightarrow$ (ii) follows immediately from the definition of Lorentzian symmetric functions, while (ii) $\Rightarrow$ (iii) follows from \cref{thm:sym-lor-supp-sufficiency}. We also note that $\supp_m(s_\lambda) = [1^d, \lambda]$, so (iii) and (iv) are equivalent. Thus, to prove the equivalence of (i) - (iv), it just remains to show that if $J = [1^d, \lambda]$, then there exists a Lorentzian symmetric function $f \in \Lambda^d$ with $\supp_m(f) = J$.

        To show this, we claim that for any $\lambda\vdash d$, $f = \sum_{\mu \preceq \lambda} \til{m}_\mu$ is a Lorentzian symmetric function. Indeed, by the argument in the proof of \cref{thm:sym-lor-supp-sufficiency}, $\supp(f_n) = \conv\{\lambda^\sigma : \sigma \in S_n\} \cap \ZZ^n$, which is exactly the integer points of a (generalized) permutahedron. Moreover, the coefficient of $x^\alpha$ in $f_n$ is exactly $\frac{1}{\alpha!}$, so $f_n$ is the generating function of an M-convex set, and hence Lorentzian by \cite[Theorem 3.10]{BrandenHuh20Lorentzian}.
        
        Finally, note that (iii) $\Leftrightarrow$ (v) follows easily from \cref{prop:radorev}.
    \end{proof}

    We note that for any $\lambda \vdash d$, $\supp_m(s_\lambda) = \supp_m(N(s_\lambda)) = [1^d,\lambda]$, so if $f$ is (normalized) Schur positive with a unique maximal element in its (normalized) Schur expansion, then it has M-convex support.
    
    In \cref{subsec:symmetric-quadratics}, we show that normalized Schur positivity is equivalent to Lorentzianity in degree two. However, for higher degrees, (normalized) Schur positivity is in general neither a necessary nor sufficient condition for Lorentzianity. In \cref{subsec:symmetric-cubics}, we show that starting in degree three, normalized Schur positivity is neither necessary nor sufficient.
    
    We also know that Schur positivity is not a sufficient condition because, in general, $s_\lambda$ is not Lorentzian, even in degree two \cite[Example 8]{huh2022schurlorentzian}. On the other hand, the normalized Schur polynomial $Ns_{33}$ is Lorentzian but not Schur positive, so Schur positivity is also not a necessary condition.

\subsection{Topology of the Space of Lorentzian symmetric Polynomials}\label{subsec:euclidean-ball}
In \cite{branden2021euclideanball}, \citeauthor{branden2021euclideanball} shows that the projective space of Lorentzian polynomials is homeomorphic to a closed Euclidean ball. In this section, we make a small adaptation to \citeauthor{branden2021euclideanball}'s proof in order to show that the space of Lorentzian symmetric polynomials is also homeomorphic to a closed Euclidean ball.
\begin{theorem}[{cf. \cite[Theorem 3.1]{branden2021euclideanball}}]\label{thm:order-two-gen-subgrp-euclidean-ball}
    Let $S \subseteq S_n$ be any subgroup of the symmetric group generated by elements of order $2$. Let $\underline{L}_S^d$ denote the set of multiaffine Lorentzian degree $d$ polynomials that are invariant under $S$. Identify $\PP \underline{L}_S^d$ with $\{f\in \underline{L}_S^d : f(1,\dots, 1) = 1\}$. Then $\PP\underline{L}_S^d$ is homeomorphic to a closed Euclidean ball.
\end{theorem}
\begin{proof}
    Borrowing the notation from \cite{branden2021euclideanball}, let $\underline{L}_n^d$ denote the set of multiaffine Lorentzian polynomials of degree $d$ on $n$ variables, let $\mathcal{L}$ be the linear operator defined by $\mathcal{L} = \frac{1}{\binom{n}{2}}\sum_{\tau} \tau$, where the sum runs over all transpositions $\tau$, and define $T_s = e^s e^{-s\mathcal{L}}$ for $s \geq 0$. By \cite[Lemma 2.4]{branden2021euclideanball}, we know that $T_s(\PP\underline{L}_n^d) \subseteq \Int(\PP\underline{L}_n^d)$ for all $s > 0$.

    We claim that for all $s > 0$ and $f \in \PP\underline{L}_S^d$, we have $T_s(f) \in \PP\underline{L}_S^d$. By \cite[Proposition 2.3]{branden2021euclideanball}, we know that $T_s(f) \in \PP\underline{L}_n^d$, so it just remains to show that $T_s(f)$ is invariant under $S$. Indeed, let $\sigma$ be a generator of $S$. By assumption $\sigma^2 = \operatorname{id}$, so $\sigma = \tau_1\cdot\dots\cdot\tau_\ell$ for some disjoint transpositions $\tau_1,\dots,\tau_\ell$. We note that for any transposition $\tau$, we have $\tau \mathcal{L} = \mathcal{L}\tau$, so
    \[\sigma T_s(f) = T_s(\tau_\ell \cdot \dots \cdot \tau_1 f) = T_s(\sigma f) = T_s(f),\]
    so $T_s(f)$ is invariant under $S$, as desired.

    In particular, for any $s> 0$, we have $T_s(\PP\underline{L}_S^d) \subseteq \Int(\PP\underline{L}_n^d) \cap \PP\underline{L}_S^d$. Since $\PP \underline{L}_S^d = \PP\underline{L}_n^d \cap H$ for some hyperplane $H$, we have $\Int(\PP \underline{L}_S^d) \supseteq \Int(\PP\underline{L}_n^d) \cap H = \Int(\PP\underline{L}_n^d) \cap \PP\underline{L}_S^d$. Therefore, $T_s(\PP\underline{L}_S^d) \subseteq \Int(\PP\underline{L}_S^d)$ for all $s > 0$, which implies that $\PP \underline{L}_S^d$ is homeomorphic to a closed Euclidean ball by the contractive flow argument in \cite[Theorem 3.1]{branden2021euclideanball} and \cite[Lemma 2.3]{galashin2018contractiveflow}.
\end{proof}
\begin{theorem}\label{thm:proj sym lor=ball}
    For all $n,d$, the projective space of Lorentzian symmetric polynomials of degree $d$ on $n$ variables is homeomorphic to a closed Euclidean ball.
\end{theorem}
\begin{proof}
    Consider the polynomial ring $\RR[x_{ij} : 1 \leq i \leq n, 1 \leq j \leq d]$, and let $S$ be the subgroup of the symmetric group generated by the transpositions $(x_{i,j}, x_{i,j'})$ for all $i \in [n], j,j'\in[d]$, as well as the order two elements of the form $\prod_{j=1}^d (x_{i,j}, x_{i',j})$ for all $1 \leq i,i' \leq n$.

    We note that $\underline{L}_S^d$ is exactly the polarizations of Lorentzian symmetric polynomials of degree $d$ in $n$ variables. Since polarization is a linear map, the set of Lorentzian symmetric polynomials is homeomorphic to $\underline{L}_S^d$, so the projective space of Lorentzian symmetric polynomials is homeomorphic to $\PP \underline{L}_S^d$ and therefore homeomorphic to a closed Euclidean ball by \cref{thm:order-two-gen-subgrp-euclidean-ball}.
\end{proof}

Similar to the projective space of Lorentzian polynomials, this ball is also naturally stratified by support. Specifically, we get one stratum for each $\lambda\vdash d$, consisting of Lorentzian symmetric polynomials whose $m$-support is $[1^d,\lambda]$.

\subsection{Hessians of Symmetric Polynomials}\label{subsec:reductions}
In this section, we examine the structures appearing in the Hessians of quadratic derivatives of symmetric polynomials. Specifically, these Hessians will have a very particular block matrix structure that makes it easy to explicitly compute some of their eigenvalues and greatly simplifies computing their signatures. The results in this section underpin our characterizations of low-degree Lorentzian symmetric polynomials in \cref{sec:small degree} as well as our simplifications in \cref{sec:simplifying-results}.

Let $f$ be a homogeneous symmetric polynomial of degree $d$ with nonnegative coefficients. We study the Lorentzian property of such polynomials $f$ by their $\widetilde{m}_\alpha$-expansions:
\[
    f=\sum_{\alpha\vdash d} c_\alpha \widetilde{m}_\alpha.
\]
Since $f$ is symmetric, we claim it suffices to analyze quadratic forms associated to all \emph{partitions} of $d-2$, rather than needing to study all compositions as we would for general Lorentzian polynomials. Indeed, we know that reordering the variables preserves Lorentzianity, and we note that for any composition $\gamma \in \ZZ_{\geq 0}^n$ and permutation $\sigma \in S_n$, we have
\[\dl^{\sigma \cdot \gamma} f = \dl^{\sigma \cdot \gamma}(\sigma\inv\cdot f) = \sigma\inv \cdot \dl^\gamma f,\]
where $\sigma\cdot\gamma = (\gamma_{\sigma(1)},\dots,\gamma_{\sigma(n)})$ and $\sigma\inv \cdot f = f(x_{\sigma\inv(1)},\dots,x_{\sigma\inv(n)})$.
Since permuting the variables does not affect the signature of the Hessian, this means that $\nabla^2 \dl^{\sigma\cdot\gamma} f$ has the same signature as $\nabla^2 \dl^{\gamma}f$. In particular, any composition of $d-2$ is a permutation of some partition $\mu\vdash d-2$, so it suffices to check that $\nabla^2 \dl^\mu f$ has Lorentzian signature for all $\mu\vdash d-2$.

We will find that for such quadratics, each Hessian is a block matrix where each off diagonal block is a scaled all 1's matrix and each diagonal block will be of the form
\[\begin{pmatrix}
    p & q & \cdots & q\\
    q & p & \ddots & \vdots \\
    \vdots & \ddots & \ddots  & q\\
    q & \cdots & q & p
\end{pmatrix}.\]
We call symmetric matrices with such a block structure \emph{Haynsworth matrices}. In this section, we show that Hessians of quadratic derivatives of symmetric polynomials are Haynsworth matrices, and that Haynsworth matrices are particularly amenable to analyzing the eigenvalue signature.
\begin{definition}
    Let $\Gamma=(P_1,\ldots, P_k)$ be a collection of intervals that partitions $[n]$. A symmetric matrix $M\in \RR^{n\times n}$ is \emph{Haynsworth of type }$\Gamma$ if $M \in \spanset(\{\mathbf{1}_{P_i}\mathbf{1}_{P_j}^T + \mathbf{1}_{P_j}\mathbf{1}_{P_i}^T: i,j=1,\dots,k\} \cup \{\diag(\mathbf{1}_{P_j}) : j = 1,\dots, k\})$.
\end{definition}

We claim that if $f \in \Lambda_n^d$, then for any $\mu \vdash d-2$, $\nabla^2 \dl^{\mu} f$ is such a Haynsworth matrix. At the heart of this fact is the following lemma.

\begin{lemma}\label{lem:composition-equivalence}
    Let $\mu,\nu \in \ZZ_{\geq 0}^n$ be compositions, and let $[\mu],[\nu]$ be the partitions obtained by ordering the parts of $\mu$ and $\nu$, respectively. If $[\mu] = [\nu]$ and $\mu_i = \nu_j$, then $[\mu + e_i] = [\nu + e_j]$.
\end{lemma}
\begin{proof}
    We note that $[\mu] = [\nu]$ if and only if their associated multisets $\lbbrace\mu\rbbrace$ and $\lbbrace\nu\rbbrace$ are equal. Further, we note that the only difference between $\lbbrace \mu\rbbrace$ and $\lbbrace \mu + e_i\rbbrace$ is that $\mu_i$ appears one fewer time while $\mu_i + 1$ appears one more time. Hence, we have
    \begin{align*}
        \lbbrace \mu + e_i\rbbrace = \lbbrace \mu\rbbrace \setminus \{\mu_i\} \cup \{\mu_i + 1\} = \lbbrace \nu\rbbrace \setminus \{\nu_j\} \cup \{\nu_j + 1\} = \lbbrace \nu + e_j\rbbrace
    \end{align*}
    and so $[\mu + e_i] = [\nu + e_j]$, as desired.
\end{proof}

\begin{lemma}~\label{lem:haynsworth blocks}
    Given $f\in\Lambda^d_n$ with coefficients $c_\alpha$ in the $\widetilde{m}_\alpha$-basis and a partition $\mu\vdash d-2$, the Hessian associated to the $\partial^\mu f$ is a Haynsworth matrix $H$, where the blocks of the Haynsworth matrix correspond to the blocks in the set partition of $[n]$ defined by $\mu$.
\end{lemma}
\begin{proof}
    We partition $\nabla^2 \dl^\mu f$ into a block matrix
    \[
                M := \nabla^2 \dl^\mu f= \left(
                \begin{array}{c|c|c|c} 
                  A_1 & B_{1,2} & \cdots & B_{1,\ell}\\ 
                  \hline 
                  B_{1,2}^T & A_2 & \ddots & \vdots \\ 
                  \hline 
                   \vdots  & \ddots  & \ddots &B_{\ell-1,\ell} \\ 
                  \hline 
                  B_{1,\ell}^T & \cdots & B_{\ell-1,\ell}^T & A_\ell
                \end{array} 
                    \right)
            \]
    where $i,j \in [n]$ are in the same block of this matrix if and only if $\mu_i = \mu_j$. We wish to show that all $B_{s,t}$ blocks are constant multiples of the all ones matrix, and that each $A_t = p_t I + q_t(\ones - I)$ for some $p_t,q_t$.

    Indeed, first suppose that $(i,j), (i',j') \in B_{s,t}$ for some $s\neq t$. Then $\mu_i = \mu_{i'}$, so by \cref{lem:composition-equivalence}, we have $[\mu + e_i] = [\mu + e_{i'}]$. Since $j,j' \in \beta_t$, we have $\mu_j = \mu_{j'}$, and since $s\neq t$, we know $j \neq i$ and $j' \neq i'$. Hence $(\mu + e_i)_j = \mu_j = \mu_{j'} = (\mu + e_{i'})_{j'}$, and therefore $[\mu + e_i + e_j] = [\mu + e_{i'} + e_{j'}]$ by \cref{lem:composition-equivalence}. Thus,
    \[M_{ij} = \dl^{\mu + e_i + e_j} f = c_{[\mu + e_i + e_j]} = c_{[\mu + e_{i'} + e_{j'}]} = M_{i'j'},\]
    as desired.

    Next, suppose that $(i,j),(i',j') \in A_t$. To prove that all off-diagonal entries of $A_t$ are equal, we need to show that if $i \neq j$ and $i'\neq j'$, then $M_{ij} = M_{i'j'}$. Indeed, we again have $\mu_i = \mu_{i'}$ and $(\mu + e_i)_j = (\mu + e_{i'})_{j'}$, so again by \cref{lem:composition-equivalence}, $[\mu + e_i + e_j] = [\mu + e_{i'} + e_{j'}]$, and hence
    \[M_{ij} = c_{[\mu + e_i + e_j]} = c_{[\mu + e_{i'} + e_{j'}]} = M_{i'j'}.\]
    Finally, if $i = j$ and $i' = j'$, then $[\mu + e_i] = [\mu + e_{i'}]$ and $(\mu + e_i)_i = \mu_i + 1 = \mu_{i'} + 1 = (\mu + e_{i'})_{i'}$ so \cref{lem:composition-equivalence} again applies, and we have
    \[M_{ii} = c_{[\mu + 2e_i]} = c_{[\mu + 2e_{i'}]} = M_{i'i'},\]
    as desired.
\end{proof}

Notably, when Haynsworth matrices carry nontrivial blocks of size greater than one, we can immediately compute a special class of eigenvalues which we call \emph{linear eigenvalues.}
\begin{lemma}~\label{lem:Tracy trick}
    Let $M$ be a Haynsworth matrix with block sizes $n_1,n_2,\ldots,n_\ell$ of the form: 
\[
    M=\left(\begin{array}{c|c|c|c} 
        A_1 & B_{1,2} & \cdots & B_{1,\ell} \\ 
        \hline 
        B_{1,2}^T & A_2 & \ddots & \vdots \\ 
        \hline 
        \vdots  & \ddots  & \ddots &B_{\ell-1,\ell} \\ 
        \hline 
        B_{1,\ell}^T & \cdots & B_{\ell-1,\ell}^T & A_\ell 
    \end{array} 
    \right)
\]
where $A_t= p_t I+q_t(\mathbbm{1}-I)$ and $B_{s,t}=b_{s,t} \mathbbm{1}$. Whenever $n_t \geq 2$, then $p_t - q_t$ is an eigenvalue of $M$ with multiplicity $n_t - 1$, giving us a total of $n - \ell$ linear eigenvalues. Moreover, the signature of the remaining eigenvalues is determined by the signature of the $\ell\times \ell$ matrix $\til{M}$ with entries
\[\til{M}_{st} = \begin{cases}
    n_s^2q_s + n_s(p_s-q_s), & s = t\\
    n_sn_tb_{s,t}, & s \neq t
\end{cases}.\]
\end{lemma}
\begin{proof}
Index the entries of the $t$-th block by $\{m_t,\ldots,m_{t+1}-1\}$ as in the previous subsection, i.e.,~$m_{t+1}-1-m_t=n_t$.  Observe that for each block $A_t$ with $n_t \geq 2$, we can immediately find $n_t-1$ eigenvectors $e_{j}-e_{m_t-1}$ with eigenvalue $p_t-q_t$ for $m_t<j<m_{t+1}.$ Thus, for $t = 1,\dots,\ell$, the $t$th block contributes $n_t - 1$ linear eigenvalues, giving us a total of $\sum_{t=1}^\ell (n_t - 1) = n-\ell$ linear eigenvalues.

Moreover, since our eigenvectors form an orthonormal basis, the remaining signature is determined by the $\ell\times \ell$ matrix $\widetilde{M}$ which is obtained by a congruence transform that collapses each of the blocks uniformly:
$$\widetilde{M}:=\left(\begin{array}{c|c|c|c}
    \begin{matrix}
        1 & \cdots & 1
    \end{matrix} & \mathbf{0} &\cdots &\mathbf{0}\\
    \hline 
    \mathbf{0} & \begin{matrix}
        1 & \cdots & 1
    \end{matrix}  & \ddots & \vdots \\
    \hline
     \vdots& \ddots & \ddots &\mathbf{0}\\
     \hline
     \mathbf{0} & \cdots& \mathbf{0} & \begin{matrix}
        1 & \cdots & 1
    \end{matrix} 
\end{array}\right) M\left(\begin{array}{c|c|c|c}
    \begin{matrix}
        1 \\ \vdots \\ 1
    \end{matrix} & \mathbf{0} &\cdots &\mathbf{0}\\
    \hline 
    \mathbf{0} & \begin{matrix}
        1 \\ \vdots \\ 1
    \end{matrix}  & \ddots & \vdots \\
    \hline
     \vdots& \ddots & \ddots &\mathbf{0}\\
     \hline
     \mathbf{0} & \cdots& \mathbf{0} & \begin{matrix}
        1 \\ \vdots \\ 1
    \end{matrix} 
\end{array}\right).$$
In particular, $\widetilde{M}$ has entries
$$\widetilde{M}_{s,t}=\begin{cases}
    n_s^2q_s-n_s(q_s-p_s) & \text{if }s=t,\\
    n_sn_t b_{st} &\text{if else.}
\end{cases}$$
\end{proof}
Note that if we wish to have a nonnegative Haynsworth matrix with Lorentzian signature, then we must require $p_t\leq q_t$ whenever $n_t \geq 2$, since $\til{M}$ will have nonnegative entries and hence at least one positive eigenvalue. In the case of a Lorentzian symmetric polynomial $f$, these linear inequalities on the quadratic Hessians are between coefficients in the $\widetilde{m}$-basis. Furthermore, by studying the $2\times 2$-minors, one finds that the coefficients respect the dominance order on partitons. 

\begin{theorem}~\label{thm: dom implies less}
    Let $f=\sum_{\alpha\vdash d}c_\alpha \widetilde{m}_\alpha$ be a Lorentzian symmetric polynomial of degree $d$, with $c_\alpha$ its $\widetilde{m}$-coefficients in the normalized monomial basis. If $\lambda,\mu\vdash d$ such that $\mu\preceq \lambda$ in dominance order, then $c_\mu\geq c_\lambda.$
\end{theorem}
\begin{proof}
    By \cref{thm:sym-lor-supp-sufficiency}, if $\mu\preceq\lambda$ and $c_\lambda > 0$, then $c_\mu > 0$, so we may assume that $c_\lambda, c_\mu > 0$. Moreover, it suffices to prove the result for covering relations $\mu \lessdot \lambda$.

    We proceed by induction on the dominance order. For the base case, the unique minimal element in dominance order is $\mu = 1^d$, and the unique covering relation over $\mu$ is $\lambda = (2,1^{d-2})$. Then we compute
    \[\dl^{(1^{d-2})} f = c_\mu \til{m}_{11} + c_\lambda(\til{m}_{2} + (x_1 + \dots + x_{d-2})\til{m}_1),\]
    where the $\til{m}$ are now symmetric functions in $x_{d-1},\dots,x_n$.
    
    The $2\times 2$ submatrix of $\nabla^2 \dl^{(1^{d-2})} f$ indexed by $x_d,x_{d+1}$ is
    \[\begin{pmatrix}
        c_\lambda & c_\mu\\
        c_\mu & c_\lambda
    \end{pmatrix},\]
    and by Lorentzianity it must have at most one positive eigenvalue, which implies that $c_\mu \geq c_\lambda$, as desired.

    For the inductive step, suppose we have $\mu \vdash d$ such that $c_\nu \geq c_\mu$ for all $\nu \preceq \mu$, and suppose that $\lambda \gtrdot \mu$. Recall that $\lambda \gtrdot \mu$ if and only if there exists $i > k$ such that $\lambda_i=\mu_i+1$, $\lambda_k=\mu_k-1$, and $\lambda_j=\mu_j$ for all $j\neq i,k$; and either 
    \begin{enumerate}[(1)]
        \item $\mu_i=\mu_k$, or
        \item $\mu_i > \mu_k$ and $k=i+1$.
    \end{enumerate}
    We obtain desired inequalities by considering $\partial^\gamma f$ where $\gamma =\mu-e_i-e_k$. We note that by construction, $\gamma + e_i + e_k = \mu$ and $\gamma + 2e_i = \mu + e_i - e_k = \lambda$. Now, consider the $2\times 2$ minor of $\nabla^2 \dl^\gamma f$ indexed by $x_i, x_k$:
    $$
    Q := \begin{pmatrix}
       \partial_i^2\partial^\gamma f & \partial_i\partial_k\partial^\gamma f\\
       \partial_i\partial_k\partial^\gamma f & \partial^2_k\partial^\gamma f
    \end{pmatrix} = \begin{pmatrix}
        c_{\lambda} & c_\mu \\
        c_\mu & c_{[\gamma + 2 e_k]}
    \end{pmatrix},
    $$
    where we use $[\cdot]$ to denote the partition associated to a composition. By Lorentzianity, this submatrix must have at most one positive eigenvalue, which happens if and only if $c_\mu^2 \geq c_{\lambda} c_{[\gamma + 2e_k]}$. We consider two cases: either $\mu_i = \mu_k$, or $\mu_i > \mu_k$ and $k = i+1$.
    
    In the former case, $\mu_i = \mu_k$, so $[\gamma + 2 e_k] = [\mu - e_i + e_k] = [\mu+e_i-e_k] = \lambda$. Therefore, since $Q$ has Lorentzian signature, we have $c_\mu^2 \geq c_\lambda^2$, and since $c_\mu, c_\lambda > 0$, we have $c_\mu \geq c_\lambda$, as desired.
    
     If $k = i+1$ and $\mu_i > \mu_{i+1}$, then we have $c_\mu^2 \geq c_{\lambda} c_\nu$, where $\nu = [\mu - e_i + e_{i+1}]$. We note that $\nu\preceq \mu$, so by induction we have $c_\nu\geq c_\mu.$ Therefore, $c_\mu^2 \geq c_\lambda c_\nu \geq c_\lambda c_\mu$, so $c_\mu \geq c_\lambda$, as desired.
\end{proof}
\begin{remark}
    If we apply the linear inequalities from \cref{thm: dom implies less} to normalized Schur polynomials, we recover the monotonicity of Kostka coefficients proved in \cite[Theorem 2]{white1980monotonicity}.
\end{remark}

From \cref{lem:haynsworth blocks,lem:Tracy trick}, we immediately have the following characterization of Lorentzian symmetric functions, which is the main theorem of this section.
\begin{theorem}\label{thm:sym lor char}
    Let $f=\sum_{\alpha\vdash d} c_\alpha \widetilde{m}_\alpha \in\Lambda^d.$ Then $f$ is a Lorentzian symmetric function if and only if 
    \begin{enumerate}
        \item[(M)] $\supp_m(f)$ has a unique maximal element in dominance order,
        \item[(D)] $c_\lambda\leq c_\mu$ whenever $\mu\preceq \lambda$ in dominance order, and
        \item[(H)] for all partitions $\mu\vdash d-2$ where \begin{itemize}
            \item $\ell(\mu)=k$,
            \item $\mu$ induces a set partition $(\beta_1,\ldots, \beta_\ell)$ of $[k]$
        \end{itemize}
        \noindent the $(\ell+1)\times (\ell+1)$ symmetric matrix $M(\mu)$ has at most one positive eigenvalue:
        \[
            M(\mu)=\left(\begin{array}{ccccc} 
                a_1 & b_{1,2} & \cdots & b_{1,\ell} &r_1\\ 
                b_{1,2} & a_2 & \ddots & \vdots &\vdots\\ 
                \vdots  & \ddots  & \ddots &b_{\ell-1,\ell} &r_{\ell-1} \\ 
                b_{1,\ell} & \cdots & b_{\ell-1,\ell} & a_\ell & r_\ell\\
                r_1 &\cdots & r_{\ell-1} & r_\ell  & q
            \end{array} 
            \right),
        \]
        where 
            \begin{align*}
                a_t &= n_t \left((c_{\mu+2e_{m_t}}+(n_t-1)c_{\mu+e_{m_t}+e_{m_t+1}}\right)\\
                b_{s,t} &= n_sn_tc_{\mu+e_{m_s}+e_{m_t}}\\
                r_s &= n_s c_{\mu+e_{m_s}+e_{k+1}}\\
                q &= c_{\mu+e_{k+1}+e_{k+2}}.
            \end{align*}
    \end{enumerate}

    Moreover, for any $f \in \Lambda^d_n$ with $n \geq d$, $f$ is a Lorentzian symmetric polynomial if and only if (M), (D), and (H') hold, where (H') is identical to (H) except $q$ is replaced by $q' = c_{\mu + e_{k+1} + e_{k+2}} - \frac{1}{n-k}(c_{\mu + e_{k+1} + e_{k+1}} - c_{\mu + 2e_{k+1}})$.
\end{theorem}
\begin{proof}
    We treat only the symmetric function case of the above theorem. The proof for symmetric polynomials follows nearly identically.

    First, if $f$ is Lorentzian, then (M) holds by \cref{lem:sym-lorentzian-unique-maxl}, and (D) holds by \cref{thm: dom implies less}. To see that (H) holds, we first apply the Haynsworth reduction technique from \cref{lem:Tracy trick}, then apply a congruence transform by $\diag(1,\dots,1,\frac{1}{n-k})$ and take the limit as $n\to\infty$, which gives the desired $M(\mu)$.

    Now, suppose that (M), (D), and (H) hold. We note that (D) implies that $\supp_m(f)$ is downward closed, so (M) and (D) together imply that $f$ has M-convex support by \cref{thm:sym-lor-supp-sufficiency}.
    
    Thus, it just remains to show that $\nabla^2 \dl^\mu f$ has Lorentzian signature for all $\mu\vdash d-2$. Indeed, we note that for any diagonal block in the Haynsworth structure described in \cref{lem:haynsworth blocks}, if $A_{t,t} = c_\nu I + c_\lambda(\ones - I)$, then $\lambda \preceq \nu$. In particular, (D) implies that all of the linear eigenvalues described in \cref{lem:Tracy trick} are nonpositive, and so $\nabla^2\dl^\mu f_n$ has Lorentzian signature if and only if
    \[Q(\mu) = \begin{pmatrix}
        a_1 & b_{1,2} & \cdots & b_{1,\ell} &(n-k)r_1\\ 
        b_{1,2} & a_2 & \ddots & \vdots &\vdots\\ 
        \vdots  & \ddots  & \ddots &b_{\ell-1,\ell} &(n-k)r_{\ell-1} \\ 
        b_{1,\ell} & \cdots & b_{\ell-1,\ell} & a_\ell & (n-k)r_\ell\\
        (n-k)r_1 &\cdots & (n-k)r_{\ell-1} & (n-k)r_\ell  & (n-k)\left((n-k-1)c_{\mu + e_{k+1}+e_{k+2}} + c_{\mu + 2e_{k+1}}\right)
    \end{pmatrix}\]
    has at most one positive eigenvalue. Since $c_{\mu+e_{k+1}+e_{k+2}} \geq c_{\mu + 2e_{k+1}}$ by (D), we note that
    \[\begin{pmatrix}
        1 & & & \\
        & \ddots & &\\
        & & 1 &\\
        & & & \frac{1}{n-k}
    \end{pmatrix}Q(\mu)\begin{pmatrix}
        1 & & & \\
        & \ddots & &\\
        & & 1 &\\
        & & & \frac{1}{n-k}
    \end{pmatrix} = M(\mu) - \frac{c_{\mu + e_{k+1}+e_{k+2}} - c_{\mu + 2e_{k+1}}}{n-k}e_{\ell+1}e_{\ell+1}^T\preceq M(\mu).\]
    Since $M(\mu)$ has Lorentzian signature by (H), $Q(\mu)$ has Lorentzian signature as well, and hence $\nabla^2\dl^\mu f_n$ has at most one positive eigenvalue for all $n$. Hence, if (M), (D), and (H) hold, then $f$ is a Lorentzian symmetric function, as desired.
\end{proof}

In \cite{chin2024realstabilitylogconcavity}, the first author shows that for fixed degree $d$, one can test if a homogeneous polynomial of degree $d$ is Lorentzian using a polynomial number of operations in the number of variables. For symmetric polynomials, we can do even better: the number of arithmetic operations is now independent of the number of variables.
\begin{theorem}\label{thm:const-time}
    For fixed degree $d$, we can test Lorentzianity of $f \in \Lambda^d$ or $f \in \Lambda^d_n$ in a constant number of arithmetic operations. Specifically, the number of operations is independent of the number of variables.
\end{theorem}
\begin{proof}
    Since $d$ is fixed, we can enumerate partitions of $d$ and $d-2$ in constant time. Checking dominance order to verify conditions (M) and (D) from \cref{thm:sym lor char} therefore takes constant time, as does computing the reduced Hessian $M(\mu)$ for each $\mu \vdash d-2$. (Note that for the symmetric polynomials case, the parameter $n$ may appear in some entries of $M(\mu)$.) If $M(\mu)$ is size $(\ell + 1)\times(\ell+1)$, then $1 + \dots + \ell \leq d-2$, so $\ell \leq O(\sqrt{d})$, which we regard as a constant. Then, by \cref{prop:signature-from-minors}, we may test the signature of each $M(\mu)$ in a constant number of operations, as desired.
\end{proof}

\section{Applying Reductions for Small Degrees} \label{sec:small degree}

In this section, we demonstrate the techniques outlined in \cref{subsec:reductions} and give explicit complete characterizations of Lorentzian symmetric polynomials of degree up to four and Lorentzian symmetric functions for degrees five and six.
    
\subsection{Symmetric Quadratics}\label{subsec:symmetric-quadratics}
    For quadratic polynomials, we only have a single Haynsworth block, so quadratic Lorentzian symmetric polynomials of degree two are characterized by a single linear inequality. In fact, quadratic symmetric functions are Lorentzian if and only if their normalized Schur expansion is nonnegative.

    \begin{theorem}\label{prop:sym-2-Lor}
        Let $f = a \til{m}_2 + b \til{m}_{11} \in\Lambda^2$. The following are equivalent:
        \begin{enumerate}[(i)]
            \item The symmetric function $f$ is Lorentzian.
            \item The symmetric polynomial $f_n \in \Lambda_n^2$ is Lorentzian for some $n \geq 2$.
            \item The coefficients of $f$ satisfy $0 \leq a \leq b$.
        \end{enumerate}
    \end{theorem}
    \begin{proof}
        The implication (i) $\Rightarrow$ (ii) holds by the definition of Lorentzian symmetric functions.
        
        To prove (ii) $\Rightarrow$ (iii), we note that by direct computation,
        \[\nabla^2 f_n = \begin{pmatrix}
            a & b & \dots & b\\
            b & a & \ddots & \vdots\\
            \vdots & \ddots & \ddots & b\\
            b & \dots & b & a
        \end{pmatrix} = b \ones\ones^T + (a-b)I,\]
        which has eigenvalues $\lambda_1 = a + (n-1) b$ with multiplicity one and $\lambda_2 = a-b$ with multiplicity $n-1$. If $f_n$ is Lorentzian, then $a, b \geq 0$ and at least one of $a,b$ is strictly positive, so $\lambda_1 > 0$. Since $n\geq 2$, $\lambda_2$ has multiplicity at least one, so we must have $\lambda_2 = a-b \leq 0$, and hence $b \geq a$, as desired.

        On the other hand, if $0 \leq a \leq b$, then $f$ has nonnegative coefficients and by the same computation as above, $\nabla^2 f_n$ has at most one positive eigenvalue for all $n$. Moreover, since $a$ and $b$ are not both zero, this implies that $b > 0$, so $f_n$ is indecomposable for all $n$. Hence, $f$ is Lorentzian, so (iii) $\Rightarrow$ (i), as desired.
    \end{proof}

    In fact, this condition coincides exactly with $f$ having a nonnegative expansion in the normalized Schur basis.

    \begin{proposition}
        Let $f$ be a symmetric polynomial of degree two in $n \geq 2$ variables.  Then $f$ is Lorentzian if and only if it is normalized Schur positive. That is,
        $f\in L^2_n\cap \Lambda_n^2$ if and only if $f=aNs_{2}+bNs_{11}$ with $a,b\geq 0.$
    \end{proposition}
    \begin{proof}
        Let $f = a N s_{2} + b N s_{11}$. We wish to show that $f$ is Lorentzian if and only if $a, b \geq 0$. We note that $N s_{11} = \til{m}_{11}$ and $N s_2 = \til{m}_2 + \til{m}_{11}$, so $f = a \til{m}_2 + (a+b)\til{m}_{11}$. Thus, by \cref{prop:sym-2-Lor}, $f$ is Lorentzian if and only if $0 \leq a \leq a+b$, which happens if and only if $a,b\geq 0$, as desired.
    \end{proof}

\subsection{Symmetric Cubics}\label{subsec:symmetric-cubics}
Next, we analyze the semialgebraic set of Lorentzian symmetric polynomials of degree three. We show that Lorentzian cubics are defined by a chain of linear inequalities, corresponding to the inequalities from \cref{thm: dom implies less}, as well as a single quadratic inequality.

The computations in \cref{thm:sym-cubic-lc-iff} are essentially those described in the structural results in \cref{lem:haynsworth blocks,lem:Tracy trick}. However, we will go through the steps in detail in the proof, so as to have an explicit illustrative example of how our techniques work.

\begin{theorem}\label{thm:sym-cubic-lc-iff}
    Let $f = a \widetilde{m}_3 + b \widetilde{m}_{21} + c \widetilde{m}_{111} \in \RR_{\geq 0}[x_0,\dots,x_n]$ be a symmetric homogeneous cubic polynomial on $n+1 \geq 3$ variables. Then $f$ is Lorentzian if and only if $0 \leq a\leq b\leq c$ and $ab + (n-1)ac \leq nb^2$.
\end{theorem}
\begin{proof}
     First, we consider the case where $b = 0$. In this case, the claim reduces to showing that $f$ is Lorentzian if and only if $a = 0$. By \cref{thm:sym-lor-supp-sufficiency}, if $f$ is Lorentzian, then $\supp_m(f)$ is an interval $[111,\mu]$ in dominance order, and if $b = 0$, then $(2,1) \notin \supp_m(f)$, which implies that $\supp_m(f) = \{111\}$ and so $a = 0$. On the other hand, if $a = 0$, then $f = c \til{m}_{111} = c e_3$, which is Lorentzian \cite[Example 2.27]{BrandenHuh20Lorentzian}, as desired.

    It now remains to consider the case where $b > 0$. We note that the fact that $(2,1) \in \supp_m(f)$ implies that $f$ is indecomposable. As noted in \cref{subsec:reductions}, all derivatives of the same type are equivalent up to permuting the variables. The only partition of $d-2 = 1$ is $(1)$, so it suffices to check that $M := \nabla^2 \dl_0 f$ has Lorentzian signature.

    By the definition of the symmetric normalized monomial basis,
    \[
    \partial_{0} f = a\frac{x_0^2}{2} + b\sum_{j=1}^n x_0x_j + b \sum_{j=1}^n \frac{x_j^2}{2} + c \sum_{1 \leq j < k \leq n} x_jx_k,
    \]
    so
    \[
    M = \left(\begin{array}{c|cccc}
        a & b & b & \dots & b\\\hline
        b & b & c & \dots & c\\
        b & c & \ddots & \ddots & \vdots\\
        \vdots & \vdots & \ddots & \ddots& c\\
        b & c & \dots & c & b
    \end{array}\right) \in \RR^{(n+1) \times (n+1)}.
    \]
    It remains to show that $M$ has at most one positive eigenvalue if and only if $a \leq b \leq c$ and $ab + (n-1)ac \leq nb^2$.
    
    The bottom right block of $M$ gives us a linear eigenvalue $\lambda_1 = b-c$ with multiplicity $n-1$ and associated eigenspace $V_1 = \spanset\{e_1-e_k : k = 2,\dots,n\}$.

    Since $\nabla^2(\partial_0 f)$ is symmetric, then the remaining eigenvectors may be chosen to be in $V_1^\perp = \spanset\{e_0, e_1 + \dots + e_n\}$. Let $U$ be the $(n+1)\times 2$ matrix with columns $\begin{bmatrix}
        e_0 & e_1 + \dots +e_n
    \end{bmatrix}$. Then the remaining eigenvectors lie in the column span of $U$, so the signs of the remaining eigenvalues are the same as the signs of the eigenvalues of $U^T M U$.

    We compute
    \begin{align*}U^T M U &= \begin{pmatrix}
        a & nb\\
        nb & n(b+(n-1)c)
    \end{pmatrix}
    \end{align*}
    Since this is a $2\times 2$ matrix with nonnegative entries, and since it is not the zero matrix, it has at least one positive eigenvalue. Moreover, it has exactly one positive eigenvalue if and only if its determinant is nonpositive, which happens if and only if
    $$ab+(n-1)ac\leq nb^2.$$
    
    Therefore, $M$ has Lorentzian signature if and only if $b \leq c$ and $ab + (n-1)ac \leq nb^2$. To complete the proof, we note that if $M$ is Lorentzian, then $a \leq b$ by \cref{thm: dom implies less}, since $(3) \succ (2,1)$ in dominance order.
\end{proof}

\begin{figure}[h]
    \centering
    \includegraphics[width=0.75\linewidth]{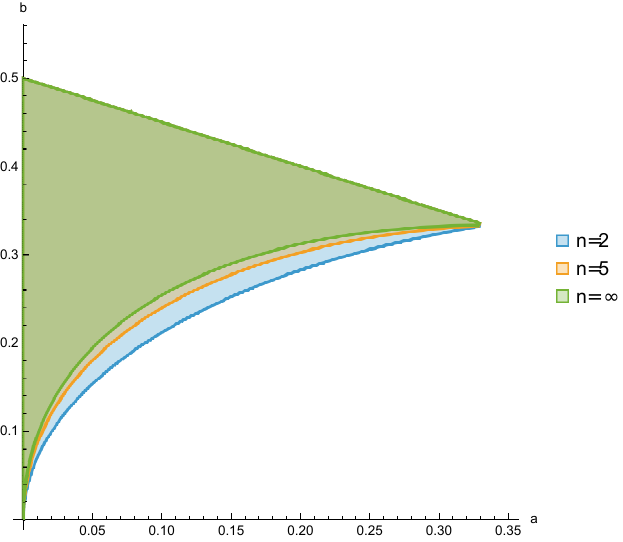}
    \caption{Regions defined by the conditions of \cref{thm:sym-cubic-lc-iff} for $n = 2, 5$ and as $n\to\infty$, intersected with the hyperplane $a+b+c = 1$. Note that each region is homeomorphic to a closed Euclidean ball, as shown in \cref{thm:proj sym lor=ball}.}
    \label{fig:symm_cubic_regions}
\end{figure}

We can then use the conditions derived in~\cref{thm:sym-cubic-lc-iff} and take $n\rightarrow \infty$ to obtain the following characterization of cubic Lorentzian symmetric functions. Interestingly, note that it suffices to check the dominance linear inequalities from \cref{thm: dom implies less} and log-concavity of coefficients with respect to dominance:
\begin{corollary}
    The symmetric function $f=a \widetilde{m}_3+b\widetilde{m}_{21}+c \widetilde{m}_{111}\in \Lambda^3$ is Lorentzian if and only if $0\leq a\leq b\leq c$ and $ac\leq b^2.$
\end{corollary}
We also note that cubic Lorentzian symmetric polynomials are Schur-positive.

\begin{corollary}
    If $f\in L^3_n\cap \Lambda^3_n$, then $f$ is $s$-positive.
\end{corollary} 
\begin{proof}
    Suppose $f = a \til{m}_3 + b\til{m}_{21} + c\til{m}_{111}$ is Lorentzian. By direct computation, we note that the expansion of $f$ in the Schur basis is
    $$f=\left(\frac{a}{6}-b+c\right)s_{1,1,1}+\frac{1}{2}\left(b -\frac{a}{3}\right)s_{2,1}+\frac{a}{6} s_3.$$
    This immediately implies $s$-positivity since $b\leq c$ and $b\geq a\geq  \frac{a}{3}$.
\end{proof}

However, unlike the quadratic case, Lorentzianity is no longer equivalent to normalized Schur positivity. For example, $4Ns_{111}+Ns_{21}+Ns_{3} = \til{m}_{3} + 2\til{m}_{21} + 7\til{m}_{111}$ is normalized Schur positive but not Lorentzian, even when restricted to three variables. Normalized Schur positivity is also no longer a necessary condition, as one can see from the corollary below.

\begin{corollary}\label{cor:cubic-lor-normalized-schur-expansion}
Let $f=a Ns_3+b Ns_{2,1}+c Ns_{1,1,1}.$ Then, $f$ is Lorentzian if and only if $a,b,b + c\geq 0$ and $ac - \frac{1}{n}a(b + c)\leq b^2.$
\end{corollary}
\begin{proof}
    Changing to the normalized monomial basis, we have
    \begin{align*}
        f &= a(\til{m}_{111} + \til{m}_{21} + \til{m}_{3}) + b(2\til{m}_{111} + \til{m}_{21}) + c(\til{m}_{111} + \til{m}_{21} + \til{m}_{3})\\
        &= a \til{m}_3 + (a + b)\til{m}_{21} + (a + 2b + c)\til{m}_{111}.
    \end{align*}

    By \cref{thm:sym-cubic-lc-iff}, $f$ is Lorentzian if and only if $0 \leq a \leq a+b \leq a+ 2b + c$ and $a(a + b) + (n-1)a(a+2b + c) \leq n(a+b)^2$. The first chain of inequalities gives us exactly $a, b, b+c \geq 0$. The latter inequality can be rearranged to $ac - \frac{1}{n}a(b + c) \leq b^2$, as desired.
\end{proof}
This implies that Lorentzian cubics need not be normalized Schur positive. For example, $f = N s_3 + 2 N s_{21} - N s_{111} = \til{m}_3 + 3\til{m}_{21} + 4\til{m}_{111}$ is not normalized Schur positive, but it satisfies the conditions of \cref{cor:cubic-lor-normalized-schur-expansion} and is therefore Lorentzian.

\subsection{Symmetric Quartics}\label{sec:symmetric-quartic}
We continue our study of low-degree polynomials by fully classifying Lorentzian symmetric polynomials of degree four.

\begin{theorem}~\label{thm:sym-lor-4}
    Let $f=a \widetilde{m}_4+b\widetilde{m}_{31}+c \widetilde{m}_{22}+d\widetilde{m}_{211}+e\widetilde{m}_{1111}\in \RR_{\geq 0}[x_1,\ldots,x_n]$ be a homogeneous symmetric quartic polynomial in $n\geq 4$ variables. Then $f$ is Lorentzian if and only if the following coefficient inequalities hold:
    \begin{itemize}
        \item[(1)] $0\leq a \leq b \leq c \leq d \leq e$
        \item[(2)] $a(c+d(n-2)) \leq (n-1)b^2$
        \item[(3)] $(b+c)(d+e(n-3)) \leq 2(n-2)d^2$
    \end{itemize}
\end{theorem}
\begin{proof}
    By \cref{thm:sym-lor-supp-sufficiency}, $f$ satisfies the support condition if and only if its $m$-support is downward closed, since dominance order on partitions of $d = 4$ is totally ordered. We note that the set of linear inequalities (1) implies this condition is satisfied. Therefore, it suffices to show that $f$ satisfies (H') if and only if inequalities (2) and (3) are satisfied.

    By definition, $f$ satisfies (H') if and only if
    \[M(2)=\begin{pmatrix}
        a & (n-1) b\\
        (n-1) b & (n-1)(c + (n-2)d)
    \end{pmatrix}\quad\text{and}\quad M(11)=\begin{pmatrix}
        \frac{1}{2}(b + c) & (n-2)d\\
        (n-2)d & (n-2)(d +(n-3) e)
    \end{pmatrix}\]
    have Lorentzian signature. This happens if and only if
    \[a(c + (n-2)d) \leq (n-1)b^2 \quad\text{and}\quad b(d+(n-3)e) \leq 2(n-2)d^2.,\]
    as desired.
\end{proof}

By taking $n\to\infty$, we also get a semialgebraic description for the set of Lorentzian symmetric functions of degree four.

\begin{corollary}\label{cor:deg-4-lor-symm-fn}
    The symmetric function $f=a \widetilde{m}_4+b\widetilde{m}_{31}+c \widetilde{m}_{22}+d\widetilde{m}_{211}+e\widetilde{m}_{1111}\in \Lambda^4$ is Lorentzian if and only if all of the following hold:
    \begin{itemize}
        \item[(1)] $0 \leq a\leq b\leq c\leq d \leq e$
        \item[(2)] $ad\leq b^2$ 
        \item[(3)] $(b+c)e\leq 2d^2$.
    \end{itemize}
\end{corollary}

\begin{example}
    Observe that there are quartic Lorentzian symmetric polynomials that do not extend to Lorentzian symmetric functions. For example, consider
    \[f = \til{m}_4 + 2\til{m}_{31} + 2\til{m}_{22} + 5\til{m}_{211} + 5\til{m}_{1111}.\]
    Then $f_n$ is Lorentzian if and only if $n \leq 4$, so some restrictions of $f$ are Lorentzian polynomials, but $f$ itself is not a Lorentzian symmetric function.
\end{example}

We now briefly examine the relationship between Lorentzian symmetric polynomials and M-concave functions. Suppose $f = \sum \frac{c_\alpha}{\alpha!}x^\alpha$. By \cite[Corollary 3.16]{BrandenHuh20Lorentzian}, we know that if $\nu_f: \alpha \mapsto \log(c_\alpha)$ is an M-concave function, then $f$ is Lorentzian. In general, the converse fails even for quadratics. Starting in degree four, the converse does not hold for symmetric polynomials.
\begin{example}
    Let
    \[f = \frac{1}{256}\til{m}_4 + \frac{1}{16}\til{m}_{31} + \frac{3}{8}\til{m}_{22} + \frac{1}{2}\til{m}_{211} + \til{m}_{1111}.\]
    The coefficients satisfy the conditions in \cref{cor:deg-4-lor-symm-fn}, so $f$ is Lorentzian.
    
    However, we claim that $\nu_f$ is not M-concave. Indeed, consider $\alpha=2e_1 + 2e_2, \beta=e_1+e_2+e_3+e_4$, so $\alpha_1 > \beta_1$. Then in order for $\nu_f$ to be M-concave, we must have
    \begin{align*}
        \nu_f(\alpha) + \nu_f(\beta) &\leq \max\{\nu_f(\alpha - e_1 + e_3) + \nu_f(\beta+e_1-e_3), \nu_f(\alpha-e_1 + e_4) + \nu_f(\beta+e_1-e_4)\}.
    \end{align*}
    We note that $\alpha - e_1 + e_3 = (1,2,1,0)$ and $\alpha - e_1 + e_4 = (1,2,0,1)$. Since $f$ is symmetric, this implies that $x^{\alpha - e_1 + e_3}$ and $x^{\alpha - e_1 + e_4}$ have the same coefficient, and hence $\nu_f(\alpha - e_1 + e_3) = \nu_f(\alpha - e_1 + e_4) = \log(c_{211})$. A similar computation shows that $\nu_f(\beta+ e_1 - e_3) = \nu_f(\beta + e_1 - e_4) = \log(c_{211})$.

    Hence, in order for $\nu_f$ to be M-concave, we must have $c_{22}c_{1111} \leq c_{211}^2$. However, $c_{22}c_{1111} = \frac{3}{8}$ while $c_{211}^2 = \frac{1}{4}$, so $\nu_f$ is not M-concave.
\end{example}

\begin{remark}
    This only starts to happen in degree four. In degree three, one can check that a symmetric cubic is Lorentzian if and only if $\nu_f$ is M-concave. This differs from the non-symmetric case, where M-concavity may fail even for quadratics, as shown in \cite{BrandenHuh20Lorentzian}.
\end{remark}
\subsection{Symmetric Quintics}
Next, we classify Lorentzian symmetric functions of degree five. Given the complexity of the inequalities, we do not treat the polynomial case separately here, though the computations are essentially identical.
\begin{theorem}\label{thm:sym-lor-5}
    Let $p= a\widetilde{m}_5+b\widetilde{m}_{41}+c\widetilde{m}_{32}+d\widetilde{m}_{311}+e\widetilde{m}_{221}+f\widetilde{m}_{2111}+g\widetilde{m}_{11111}$ be a symmetric function of degree five with nonnegative coefficients. Then $f$ is  Lorentzian if and only if the following coefficient inequalities hold:
    \begin{itemize}
        \item[(1)] $0 \leq a \leq b\leq c\leq d\leq e\leq f\leq g$ 
        \item[(2)] $ad\leq b^2$
        \item[(3)] $(d+2e)g\leq 3f^2$ 
        \item[(4)] $bf\leq d^2$
        \item[(5)] $ cf\leq e^2$
        \item[(6)] $\det\begin{pmatrix}
            b & c & d\\
            c & c & e\\
            d & e & f\\
        \end{pmatrix}\geq 0$
    \end{itemize}
\end{theorem}
\begin{proof}
    We again apply \cref{thm:sym lor char}. Since the dominance order on partitions of $d=5$ is a total order, $\supp_m(f)$ vacuously must have a unique maximal element. Moreover, inequality (1) is exactly condition (D). Thus, it just remains to check the Hessian condition on $\partial^3f$, $\partial^{21} f$, and $\partial^{111} f$. Condition (H) says that $f$ is Lorentzian if $M(3),M(21),M(111)$ all have Lorentzian signature. Using the definition of $M(\mu)$, we have
    \begin{align*}
        M(3)&=\begin{pmatrix}
            a & b\\
            b & d
        \end{pmatrix} & 
        M(111)&= \begin{pmatrix}
            3(d+2e) & 3f\\
            3f & g
        \end{pmatrix} &
        M({21}) &= \begin{pmatrix}
            b & c & d\\
            c & c & e\\
            d & e & f
        \end{pmatrix}
    \end{align*}
    Taking the determinants of $M(3)$ and $M(111)$ give inequalities (2) and (3), so it just remains to study the signature of $M(21)$.
    
    By \cref{prop:signature-from-minors}, $M(21)$ has Lorentzian signature if and only if the following four inequalities all hold:
    \begin{enumerate}[(i)]
        \item $b^2-bc\leq 0$
        \item $bf\leq d^2$
        \item $cf\leq e^2$
        \item $\det(M(21))\geq 0$
    \end{enumerate}
    Already (i) fits into the linear chain given in (1), and the remaining inequalities are exactly conditions (4)-(6), which completes the proof.
\end{proof}

\subsection{Symmetric Sextics}
As our final low-degree case, we fully characterize Lorentzian symmetric functions of degree six.

\begin{theorem}~\label{thm:sym-lor-6}
    Let $f = \sum_{\alpha \vdash 6} c_\alpha \til{m}_\alpha$ be a symmetric function of degree six with nonnegative coefficients. Then $f$ is Lorentzian if and only if all of the following hold:
    \begin{itemize}
        \item Whenever $\lambda \preceq \mu$ in dominance order, we have $c_\mu \geq c_\lambda$.
        \item $c_6 c_{411} \leq c_{51}^2$
        \item $c_{2211}(c_{42} + c_{33}) \leq 2 c_{321}^2$
        \item $c_{1^6}(c_{3111} + 3c_{2211}) \leq 4 c_{21111}^2$
        \item Each of the following matrices has exactly one positive eigenvalue\footnote{By \cref{prop:signature-from-minors}, testing the signature of either of these matrices reduces to checking three quadratic inequalities and one cubic inequality, but for the sake of compactness and clarity, we leave it in matrix form.}
            \begin{align*}
                Q_{31} &= \begin{pmatrix}
                    c_{51} & c_{42} & c_{411}\\
                    c_{42} & c_{33} & c_{321}\\
                    c_{411} & c_{321} & c_{311} 
                \end{pmatrix} &
                Q_{211} &= \begin{pmatrix}
                    c_{411} & c_{321} & c_{3111}\\
                    c_{321} & \frac{1}{2}(c_{321} + c_{222}) & c_{2211}\\
                    c_{3111} & c_{2211} & c_{21111}
                \end{pmatrix}.
            \end{align*}
    \end{itemize}
\end{theorem}
\begin{proof}
    By \cref{thm:sym lor char}, we know that $f$ is Lorentzian if and only if it satisfies conditions (M), (D), and (H). We note that (D) is also a condition above, so it remains to consider (M) and (H). After applying our reduction, we are left with
    \begin{gather*}
        M(4) = \begin{pmatrix}
            c_6 & c_{51}\\
            c_{51} & c_{411}
        \end{pmatrix} \qquad\qquad M(31) = \begin{pmatrix}
            c_{51} & c_{42} & c_{411}\\
            c_{42} & c_{33} & c_{321}\\
            c_{411} & c_{321} & c_{3111}
        \end{pmatrix}\\
        M(22) = \begin{pmatrix}
            c_{32} + c_{42} & 2c_{321}\\
            2c_{321} & c_{2211}
        \end{pmatrix} \qquad\qquad M(211) = \begin{pmatrix}
            c_{411} & 2c_{321} & c_{3111}\\
            2c_{321} & c_{321} + c_{222} & 2c_{2211}\\
            c_{3111} & 2c_{2211} & c_{21111}
        \end{pmatrix}\\
        M(1111) = \begin{pmatrix}
            4(3c_{2211} + c_{3111}) & 4c_{21111}\\
            4c_{21111} & c_{1^6}
        \end{pmatrix}
    \end{gather*}
    
    First, suppose $f$ is Lorentzian. The determinants of $M(4), M(22)$, and $M(1111)$ give the desired quadratic inequalities. Moreover, $Q_{31} = M(31)$, and $Q_{211}$ is obtained from $M(211)$ by scaling the second row and column by $1/2$. Hence $Q_{31}$ and $Q_{211}$ both have Lorentzian signature, as desired.

    On the other hand, suppose that the conditions above hold. We note that (D) holds by assumption. Moreover, as noted above, the quadratic and matrix inequalities exactly give us (H). Thus, it just remains to show that $\supp_m(f)$ has a unique maximal element in dominance order. Indeed, the only incomparable partitions of six are $\{(4,1,1), (3,3)\}$ and $\{(3,1,1,1), (2,2,2)\}$. If the former pair are both maximal in $\supp_m(f)$, then
    \[\det(Q_{31}) = \det\begin{pmatrix}
        0 & 0 & c_{411}\\
        0 & c_{33} & c_{321}\\
        c_{411} & c_{321} & c_{311}
    \end{pmatrix} = -c_{411}^2c_{33} < 0,\]
    so $Q_{31}$ does not have Lorentzian signature. Similarly, if the latter pair are both maximal, then
    \[\det(Q_{211}) = \det\begin{pmatrix}
        0 & 0 & c_{3111}\\
        0 & \frac{1}{2}c_{222} & c_{2211}\\
        c_{3111} & c_{2211} & c_{21111}
    \end{pmatrix} = -\frac{1}{2}c_{3111}^2c_{222} < 0,\]
    so $Q_{211}$ does not have Lorentzian signature. Hence, if $\supp_m(f)$ does not have a unique maximal element, then either $Q_{31}$ or $Q_{211}$ does not have Lorentzian signature, as desired.
\end{proof}

One surprising observation is that the condition in \cref{thm:sym lor char} that $\supp_m(f)$ has a unique maximal element is redundant in the degree six case, which is the smallest degree in which it is not vacuously satisfied. It is not known if this condition continues to be redundant in higher degrees.

\section{Applications to Known Results}\label{sec:simplifying-results}
In addition to fully characterizing Lorentzian symmetric polynomials in low degrees, our techniques also allow us to significantly simplify some known proofs that certain symmetric functions are Lorentzian. In \cref{subsec:chromatic-symmetric}, we simplify the linear algebra in the proof for chromatic symmetric functions, originally proved in \cite{Matherne24chromaticsymmetric}. We also give a simplified combinatorial proof for two-column Schur functions. The full Schur function case was proved in \cite{huh2022schurlorentzian} using volume polynomials of varieties. The two-column case was proved combinatorially in \cite{qinthesis}, and we give a simplified version of this proof in \cref{subsec:two-column-schur}.

\subsection{Lorentzianity of Chromatic Symmetric Functions}\label{subsec:chromatic-symmetric}
In \cite[Theorem 6.8]{Matherne24chromaticsymmetric}, \citeauthor{Matherne24chromaticsymmetric} prove that the chromatic symmetric functions of indifference graphs of abelian Dyck paths are Lorentzian. In this section, we use the computational reductions found in \cref{subsec:reductions} to give a shorter proof of this fact. Note that we still use the same underlying combinatorics and inequalities, but we are now able to forgo Schur complements entirely and to simplify the linear algebra considerably.

We begin by recalling some definitions.
\begin{definition}
    A \emph{Dyck path} $d$ of length $n$ is a lattice path from $(0,0)$ to $(n,n)$ that does not go below the diagonal $y = x$. For a Dyck path $d$, we define the indifference graph $G(d)$ as the graph with vertex set $[n]$ where there is an edge between $i$ and $j$, with $i < j$ if the square in column $i$ and row $n+1-j$ is between the path and the diagonal. We say that $d$ is abelian if $G(d)$ is cobipartite.
\end{definition}

\begin{definition}
    For a graph $G$ on vertex set $[n]$, the \emph{chromatic symmetric polynomial} of $G$ is
    \[X_G := \sum_{\substack{\kappa : V(G) \to \ZZ_+\\\text{proper coloring}}} x_{\kappa(1)}x_{\kappa(2)}\dots x_{\kappa(n)}.\]
\end{definition}

In \cite{Matherne24chromaticsymmetric}, \citeauthor{Matherne24chromaticsymmetric} prove that chromatic symmetric polynomials of the indifference graphs of abelian Dyck paths are Lorentzian.
\begin{theorem}[{\cite[Section 6.2]{Matherne24chromaticsymmetric}}]
    Let $d$ be an abelian Dyck path. Then $X_{G(d)}$ is a Lorentzian symmetric function.
\end{theorem}
\begin{proof}
    Following \cite[Theorem 6.8]{Matherne24chromaticsymmetric}, we know that 
    \[X_{G(d)} = \sum_{i}i! (n-2i)! r_i m_{2^i 1^{n-2i}},\]
    where $r_i$ is the number of non-attacking rook placements on the Ferrers board $B_\mu$ associated to $G(d)$.

    Since $X_{G(d)}$ has degree at most two in each variable, it suffices to consider $\dl^\alpha X_{G(d)}$ for all partitions $\alpha$ of the form $\alpha = (2^{k-1}, 1^{n-2k})$.

    Partition the variables into sets $Y_1, Y_2, Y_3$, where $Y_1 = \{x_1,\dots,x_{n-2k}\}, Y_2 = \{x_{n-2k+1}, \dots, x_{n-k+1}\}, Y_3 = \{x_{n-k+2},\dots,x_n\}$, and consider $\dl_{Y_1}\dl_{Y_3}^2 X_{G(d)}$. By direct computation,
    \begin{multline*}
         \dl_{Y_1}\dl_{Y_3}^2 X_{G(d)}
        = 2^{k+1}(k+1)!(n-2k-2)! r_{k+1}m_{11}(Y_1)\\
            + 2^{k-1} k!(n-2k)! r_k(m_2(Y_2)+ 2m_1(Y_1)m_1(Y_2)) \\
            + 2^{k-1}(k-1)!(n-2k+2)! r_{k-1}m_{11}(Y_2),
    \end{multline*}
    so
    \[\nabla^2 \dl^{\alpha} X_{G(d)} = \left(\begin{array}{ccc|ccc}
        0 & & a & & &\\
        & \ddots & & & b & \\
        a & & 0 & & &\\\hline
        & & & b & & c\\
        & b & & & \ddots &\\
        & & & c & & b
    \end{array}\right),\]
    where
    \begin{align*}
        a &= 2^{k+1}(k+1)!(n-2k-2)!r_{k+1}\\
        b &= 2^k k!(n-2k)!r_k\\
        c &= 2^{k-1} (k-1)!(n-2k+2)! r_{k-1}.
    \end{align*}
    Using our block matrix techniques, if $X_{G(d)}$ has $r$ variables, this matrix has Lorentzian signature if and only if $a \geq 0$, $b \leq c$, and
    \[\det\begin{pmatrix}
        (n-2k)(n-2k-1)a & (n-2k)(r-n+k+1)b\\
        (n-2k)(r-n+k+1) b & (r-n+k)(r-n+k+1)c + (r-n+k+1)b
    \end{pmatrix} \leq 0,\]
    which happens if and only if
    \[(n-2k-1)a(b + (r-n+k)c) - (n-2k)(r-n+k+1)b^2 \leq 0.\]
    These inequalities are verified in \cite[Propositions 6.10, 6.12]{Matherne24chromaticsymmetric} using properties of rook numbers on Ferrers boards.
\end{proof}

\subsection{Normalized Two-Column Schur Functions}\label{subsec:two-column-schur}
In \cite{qinthesis}, the second author gives a combinatorial proof that two-column Schur functions are denormalized Lorentzian. In this section we give a simplified version of this proof, again eliminating the need for Schur complements while proving the same underlying inequalities.

First, we recall the \emph{ballot numbers} $C_{k,\ell}$ where $$C_{k,\ell}=\frac{(k+\ell)!}{\frac{(k+1)!}{k-\ell+1}\cdot \ell!}=\frac{k-\ell+1}{k+1}{k+\ell\choose \ell}={k+\ell\choose \ell}-{k+\ell\choose \ell-1}.$$
By the hook length formula, $C_{k,\ell}$ is the number of standard Young tableaux on shape $(k,\ell)$ and also the number of standard Young tableaux on $(k,\ell)'$.

\begin{theorem}~\label{ex:n}
    Let $s\geq t$ such that $s +t = d$, and let $\gamma=(s,t)'$ be the two columned tableau with columns of length $s$ and $t$. Then $Ns_\gamma$ is Lorentzian.
\end{theorem}
\begin{proof}
    Since Schur polynomials of two-column tableaux have degree at most two in each variable, it is sufficient to consider $\mu=(2^p,1^q,0^m)\in\NN^n$ where $\mu\vdash d-2.$ For such a partition $\mu$, define $\ell = t - p$ and $k = s - p$.
    
    Choose variables $z_1<\ldots<z_p<x_1<\ldots<x_q<y_1<\ldots<y_m$. For ease of notation, we will denote $X = \{x_i\}_{i=1}^q$, $Y = \{y_i\}_{i=1}^m$, and $Z = \{z_i\}_{i=1}^p$. By direct computation,
    \[\partial^\mu Ns_{(s,t)'}=a\sum_{x_i<x_j\in X}x_ix_j+b\sum_{x_i\in X\\
    y_j\in Y} x_iy_j+c_1\sum_{y_i< y_j\in Y}y_iy_j+c_2\sum_{y_i\in Y}\frac{y_i^2}{2}\]
    for some $a,b,c_1,c_2$.

    First, we find the inequalities necessary to show that $\nabla^2 \dl^\mu N s_\gamma$ has Lorentzian signature. By direct computation, the Hessian is of the form 
    $$H=\left( 
    \begin{array}{c|c} 
      0 & 0  \\ 
      \hline 
      0 & H'
    \end{array} 
    \right)\quad\text{with}\quad 
    H'=\left( 
    \begin{array}{c|c} 
      A & B \\ 
      \hline 
      B^T & D
    \end{array} 
    \right)$$
    where
    \begin{align*}
        A &= a(\mathbf{1}_{q\times q} - I)\\
        B &= b \cdot \mathbf{1}_{q \times m}\\
        D &= c_1 \cdot \mathbf{1}_{m \times m} + (c_2-c_1) I
    \end{align*}
    
    Using the techniques from \cref{subsec:reductions}, $H'$ has Lorentzian signature for all $n$ if and only if the linear eigenvalues $\lambda = -a$ and $\lambda = c_2-c_1$ are nonpositive and the remaining matrix
    \[\til{H} = \begin{pmatrix}
        q(q-1) a & q b\\
        q b & c_1
    \end{pmatrix}\]
    has Lorentzian signature. Therefore, in order to show that $\nabla^2 \dl^\mu N s_\lambda$ has Lorentzian signature, it just remains to show
    \begin{align}
        c_2 &\leq c_1 \label{eq:two-col-linear}\\
        (q-1)ac_1 &\leq qb^2. \label{eq:two-col-det}
    \end{align}
    
    We note that if $\ell = 0, 1$, then $q+2 > t$ and hence $[\mu + e_{x_i} + e_{x_j}] = (2^{q+2}, 1^{p-2}) \succ \gamma = (2^t, 1^s)$, so $a = 0$. Thus, if $\ell = 0,1$, then \labelcref{eq:two-col-det} holds automatically. If $\ell = 0$, then additionally $[\mu + e_{x_i} + e_{y_j}] = (2^{q+1},1^{p}) \succ \gamma$, so $b = 0$, in which case the former inequality holds as well. Thus, it just remains to verify that $c_1 \geq c_2$ whenever $\ell \geq 1$, and that $(q-1)ac_1 - qb^2 \leq 0$ whenever $\ell \geq 2$.

    To show this, we enumerate the coefficients. Note that $a$ is the number of semistandard fillings of $(s,t)'$ with content $\mu+e_{x_1}+e_{x_2}.$ Given any such filling $f$, it must satisfy the following: 
    \begin{itemize}
        \item The first $2\times p$ boxes contain exactly the $z_i's$,
        \item There are two $x_1$ boxes and two $x_2$ boxes, and these boxes must be next to each other and immediately below the $Z$ boxes, so the next two rows of $f$ must be a row of two $x_1$s and a row of two $x_2$s.
    \end{itemize}
    Therefore, the first $p+2$ rows of such a semistandard tableau are forced. Moreover, each of $x_3,\dots,x_q$ must appear at least once in the tableau. These $q-2$ boxes plus the $2(p+2)$ boxes from the first rows total to $2(p+2) + (q-2) = 2p + q + 2 = d$ boxes, so the remaining boxes of the tableau after the first $p+2$ rows are exactly a standard Young tableau of shape $(s-(p+2),t-(p+2))' = (k-2, \ell-2)'$. Therefore, by the hook length formula,
    \[a = C_{k-2,\ell-2}.\]
    
    To compute $b$, we compute the number of semistandard Young tableaux of type $\mu + e_{x_1} + e_{y_1}$. By the same logic as above, the first $p+1$ rows are forced to be filled with $(z_1,z_1),\dots,(z_p,z_p),(x_1,x_1)$, and each of $x_2,\dots,x_d$ must appear at least once, and $y_1$ must appear exactly once. This sums to $2(p+1) + (q-1) + 1 = d$, so the remaining part of the tableau is a standard Young tableau of shape $(s-(p+1), t-(p+1))' = (k-1,\ell-1)'$. Thus,
    \[b = C_{k-1,\ell-1}.\]
    
    By similar logic, $c_1$ is the number of semistandard Young tableaux of type $\mu + e_{y_1} + e_{y_2}$. Such a tableau must have the first $p$ rows filled with the $z_i$s, and the remaining boxes form a standard Young tableau of shape $(2^{s-p},1^{t-p})$, so
    \[c_1 = C_{k,\ell}.\]
    
    Finally, $c_2$ is the number of semistandard Young tableaux of shape $(s,t)'$ with content $\mu + 2 e_{y_2} = (2^{p}, 1^q,2)$, which, by symmetry of Schur polynomials, must be the same as the number of tableaux of type $\mu + e_{x_1} + e_{y_1} = (2^p, 2, 1^{q-1},1)$. Therefore, $c_2 = b$, so
    \[c_2 = C_{k-1,\ell-1}.\]
    
    By direct computation, for $\ell \geq 1$,
    \[C_{k,\ell}=\frac{(k+\ell)(k+\ell-1)}{(k+1)\ell}C_{k-1,\ell-1}.\]

    Since $k \geq \ell \geq 1$, we have
    \begin{align*}
        \frac{k+\ell - 1}{\ell} &\geq 1  & \text{and} && \frac{k+\ell}{k+1} & \geq 1,
    \end{align*}
    so $C_{k,\ell} \geq C_{k-1,\ell-1}$. Hence, $c_1 \geq c_2$ whenever $\ell \geq 1$, as desired.
    
    It remains to verify that $(q-1)ac_1 - qb^2 \leq 0$ whenever $\ell \geq 2$. By definition, $k = s - p$, $\ell = t - p$, $s + t = d$, and $2p + q = d - 2$, so $q = k + \ell - 2$. Thus, \labelcref{eq:two-col-det} can be equivalently written as
    \[(k+\ell-3)C_{k,\ell}C_{k-2,\ell-2} - (k+\ell -2)C_{k-1,\ell-1}^2 \leq 0.\]
    Since $k \geq \ell \geq 2$, then $C_{k-2,\ell-2},C_{k-1,\ell-1},C_{k,\ell}> 0$, so it suffices to show that
    \[\frac{k+\ell-2}{k+\ell-3}\frac{C_{k-1,\ell-1}^2}{C_{k,\ell} C_{k-2,\ell-2}} \geq 1.\]
    By the same computation as above, we have
    \begin{align*}
        \frac{C_{k-1,\ell-1}}{C_{k,\ell}} &= \frac{(k+1)\ell}{(k+\ell)(k+\ell-1)} & \frac{C_{k-1,\ell-1}}{C_{k-2,\ell-2}} &= \frac{(k+\ell -2)(k+\ell -3)}{k(\ell - 1)},
    \end{align*}
    so this reduces to showing that
    \begin{equation}\label{eq:reduced-two-col-hessian-cond}
        (k+1)\ell(k+\ell -2)^2 \geq k(\ell-1)(k+\ell)(k+\ell-1).
    \end{equation}
    
    Indeed, we note that
    \begin{equation}\label{eq:two-col-hessian-cond-pt1}
        (k+1)(k+\ell -2) = k(k+\ell-1) - k + (k+\ell-2) = k(k+\ell - 2) + \ell - 2 \geq k(k+\ell-2),
    \end{equation}
    where the final inequality comes from the assumption that $\ell \geq 2$.
    
    Similarly,
    \begin{equation}\label{eq:two-col-hessian-cond-pt2}
        \ell(k+\ell-2) = (\ell-1)(k+\ell) + (k + \ell) - 2\ell = (\ell-1)(k+\ell) + k - \ell \geq (\ell-1)(k+\ell).
    \end{equation}
    
    Multiplying the inequalities from \labelcref{eq:two-col-hessian-cond-pt1,eq:two-col-hessian-cond-pt2} gives the desired inequality in \labelcref{eq:reduced-two-col-hessian-cond}, which completes the proof.
\end{proof}

\section{Nonresults: Symmetric Function Operations and Lorentzianity}\label{sec:non-results}
One of the main breakthroughs in the study of Lorentzian symmetric polynomials is that normalized Schur polynomials are Lorentzian \cite{huh2022schurlorentzian}. In the study of symmetric functions, Schur polynomials behave particularly nicely under operations such as the $\omega$-involution and Hall inner product. Inspired by this, it is natural to ask whether this nice behavior extends to all symmetric denormalized Lorentzian polynomials, but unfortunately, the answer is no.

\subsection{$\omega$-involution is not Positivity Preserving}
The $\omega$-involution is an algebra endomorphism $\Lambda \to \Lambda$ defined by $\omega(e_\lambda) = h_\lambda$, where $e_\lambda$ and $h_\lambda$ are the elementary symmetric function and complete homogeneous symmetric function indexed by a partition $\lambda$, respectively. This map is an involution, and sends Schur polynomials to Schur polynomials.
\begin{theorem}[{\cite[Theorem 7.14.5]{stanley1999ec2}}]
    For every partition $\lambda$, we have $\omega(s_\lambda) = s_{\lambda'}$.
\end{theorem}

Since $N(s_\mu)$ is Lorentzian for all partitions $\mu$, and since $N \circ \omega \circ N\inv: N(s_\lambda) \mapsto N(s_{\lambda'})$, it is reasonable to conjecture that $N \circ \omega \circ N\inv$ might preserve Lorentzianity (or, equivalently, that $\omega$ might preserve denormalized Lorentzianity), but this is not the case.

\begin{proposition}
    The operator $N\circ \omega \circ N\inv$ is not Lorentzian-preserving. In fact, it does not preserve nonnegativity of coefficients when applied to Lorentzian polynomials.
\end{proposition}
\begin{proof}
    Consider the symmetric function $f = \til{m}_{211} + \til{m}_{1111}$. We have that $f$ is Lorentzian by the characterization of Lorentzian symmetric quartics in \cref{sec:symmetric-quartic}.

    However,
    \[N(\omega(N\inv(f))) = \til{m}_{1111} - \til{m}_{22} - \til{m}_{31} - 2\til{m}_4,\]
    so $N(\omega(N\inv(f)))$ does not have nonnegative coefficients and therefore cannot be Lorentzian.
\end{proof}

\subsection{$\langle -,Ns_\lambda\rangle$ is not Denormalized Lorentzian Preserving}

One tempting way to find a combinatorial proof that normalized Schur polynomials are Lorentzian is via an alternative form of the dual Cauchy identity. The following lemma is widely known in the symmetric functions community; for example, it is used in the original proof of the Lorentzianity of normalized Schur polynomials in \cite{huh2022schurlorentzian}, cited as \cite[Exercise 15.50]{fulton1991reptheory}. For the sake of completeness, we produce the proof below.
\begin{lemma}
    For all $m,n \in \ZZ_+$, we have
    \[\prod_{i=1}^n \prod_{j=1}^m (x_i+y_j) = \sum_{\substack{\lambda : \ell(\lambda) \leq n\\ \lambda_1 \leq m}} s_\lambda(x) s_{\til{\lambda'}}(y),\]
    where $\til{\lambda'}$ denotes the complementary partition to $\lambda'$ in the $n\times m$ rectangle, i.e. $\til{\lambda'}_k = n - \lambda'_{m+1-k}$.
\end{lemma}
\begin{proof}
    We recall the dual Cauchy identity \cite[Theorem 7.14.3]{stanley1999ec2}:
    \[\prod_{i,j} (1 + x_iy_j) = \sum_{\lambda} s_\lambda(x) s_{\lambda'}(y).\]
    By specializing $x_i = y_j = 0$ for all $i >n, j > m$, we have
    \begin{align*}
        \prod_{i=1}^n \prod_{j=1}^m (1 + x_i y_j) &= \sum_\lambda s_{\lambda}(x_1,\dots,x_n,0,\dots) s_{\lambda'}(y_1,\dots,y_m,0,\dots)\\
            &= \sum_{\substack{\lambda : \ell(\lambda) \leq n\\ \lambda_1 \leq m}} s_{\lambda}(x_1,\dots,x_n) s_{\lambda'}(y_1,\dots,y_m).
    \end{align*}
    For compactness, we will omit $\ell(\lambda) \leq n$ and $\lambda_1 \leq m$ from the sum going forward.
    
    Using the above identity, we find
    \begin{align*}
        \prod_{i=1}^n \prod_{j=1}^m(x_i + y_j) &= \prod_{i=1}^n\prod_{j=1}^m y_j(1 + x_i y_j\inv)\\
        &=\left(\prod_{j=1}^m y_j^n\right)\prod_{i=1}^n\prod_{j=1}^m(1 + x_i y_j\inv)\\
        &= \left(\prod_{j=1}^m y_j^n\right)\sum_{\lambda}s_\lambda(x_1,\dots,x_n) s_{\lambda'}(y_1\inv, \dots, y_m\inv)\\
        &= \sum_{\lambda} s_\lambda(x_1,\dots,x_n)\left(y_1^n\dots y_m^n s_{\lambda'}(y_1\inv,\dots,y_m\inv)\right).
    \end{align*}

    Thus, in order to prove the desired identity, it just remains to show that for any $\lambda$ with $\ell(\lambda) \leq m$ and $\lambda_1 \leq n$, we have
    \[s_{\til{\lambda}}(y_1,\dots,y_m) = y_1^n\dots y_m^n s_\lambda(y_1\inv,\dots,y_m\inv),\]
    where $\til{\lambda}$ is the partition defined by $\til{\lambda}_k = n - \lambda_{m+1-k}$.

    Indeed, we recall the bialternant formula \cite[Theorem 7.15.1]{stanley1999ec2}:
    \[s_\lambda(x_1,\dots,x_n) = \frac{\det\left(x_i^{\lambda_j + n - j}\right)_{1\leq i,j \leq n}}{\det\left(x_i^{n-j}\right)_{1\leq i,j\leq n}}.\]
    Using this formula, we have
    \begin{align*}
        y_1^n\dots y_m^n s_\lambda(y_1\inv,\dots,y_m\inv) &= \frac{y_1^{n+m-1}\dots y_m^{n+m-1}\det\left(y_i^{-\lambda_j - m + j}\right)_{1 \leq i,j \leq m}}{y_1^{m-1}\dots y_m^{m-1}\det\left(y_i^{-n + j}\right)_{1 \leq i,j \leq m}}\\
            &= \frac{\det\left(\left(y_i^{-\lambda_j - m + j}\right)_{1 \leq i,j\leq m}\diag\left(y_i^{n+m-1}\right)\right)}{\det\left(\left(y_i^{- m + j}\right)_{1 \leq i,j\leq m}\diag\left(y_i^{m-1}\right)\right)}\\
            &= \frac{\det\left(y_i^{n-\lambda_j + j - 1}\right)}{\det\left(y_i^{j-1}\right)} = \frac{\det\left(y_i^{\til{\lambda}_{m+1-j} + m - (m+1-j)}\right)}{\det\left(y_i^{m-(m+1-j)}\right)}.
    \end{align*}
    Let $\sigma$ be the permutation on $[m]$ that transposes $j$ with $m+1-j$ for all $j$, and let $P_\sigma$ be its permutation matrix. Then, by construction, $\left(y_i^{\til{\lambda}_{m+1-j} + m - (m+1-j)}\right) = \left(y_i^{\til{\lambda}_j + m - j}\right) \cdot P_\sigma$ and similarly $\left(y_i^{m-(m+1-j)}\right)_{1\leq i,j \leq m} = \left(y_i^{m-j}\right)\cdot P_\sigma$. Plugging this into the equation above, we get
    \begin{align*}
        y_1^n\dots y_m^n s_\lambda(y_1\inv,\dots,y_m\inv) &= \frac{\det\left(y_i^{n-\lambda_j + j - 1}\right)}{\det\left(y_i^{j-1}\right)} = \frac{\det\left(y_i^{\til{\lambda}_{m+1-j} + m - (m+1-j)}\right)}{\det\left(y_i^{m-(m+1-j)}\right)}\\
            &= \frac{\det\left(\left(y_i^{\til{\lambda}_j + m - j}\right) \cdot P_\sigma\right)}{\det\left(\left(y_i^{m-j}\right)\cdot P_\sigma\right)}\\
            &= \frac{\det\left(y_i^{\til{\lambda}_j + m - j}\right)_{1\leq i,j\leq m}}{\det\left(y_i^{m-j}\right)_{1 \leq i,j \leq m}}\\
            &= s_{\til{\lambda}}(y_1,\dots,y_m),
    \end{align*}
    which completes the proof.
\end{proof}

Using the fact that Schur polynomials are an orthonormal basis with respect to the Hall inner product, this gives us a tantalizing approach towards proving that Schur polynomials are denormalized Lorentzian. We note that $\prod_{i,j} (x_i + y_j)$ is a product of linear forms, so it is real stable and therefore Lorentzian \cite[Proposition 2.2]{BrandenHuh20Lorentzian}. Moreover, by the orthonormality of Schur polynomials, we have
\[s_\lambda(x) = \left\langle \prod_{i,j} (x_i + y_j), s_{\til{\lambda'}}(y)\right\rangle,\]
so if $N \circ \langle -, s_\lambda(y)\rangle: \Lambda(x,y) \to \Lambda(x)$ were to preserve Lorentzianity, this would be sufficient to prove that Schur polynomials are denormalized Lorentzian.

Unfortunately, $\langle -, s_\lambda(x)\rangle$ does not, in general, send Lorentzian polynomials to denormalized Lorentzian polynomials.

\begin{example}
    Consider
    \begin{multline*}
        f = \frac{1}{2}m_{211}(\vec{x})m_1(\vec{y}) + m_{1111}(\vec{x})m_1(\vec{y})\\
            +\frac{1}{2}m_{21}(\vec{x})m_{11}(\vec{y}) + \frac{1}{2}m_{111}(\vec{x})m_2(\vec{y}) + m_{111}(\vec{x})m_{11}(\vec{y})\\
            + \frac{1}{2}m_{11}(\vec{x})m_{21}(\vec{y}) + m_{11}(\vec{x})m_{111}(\vec{y}).
    \end{multline*}

    Since $\langle m_{21}, s_{111}\rangle = -2$ and $\langle m_{111},s_{111}\rangle = 1$, 
        \[\langle f, s_{111}(\vec{x})\rangle = -2\left(\frac{1}{2}m_{11}(\vec{y})\right) + \frac{1}{2}m_2(\vec{y}) + m_{11}(\vec{y}) = \frac{1}{2}m_2(\vec{y}),\]
    so $\langle f, s_{111}(\vec{x})\rangle$ is not denormalized Lorentzian.

    However, $f$ is Lorentzian. The quadratic derivatives of $f$ that are not identically zero are as follows. Note that we abuse notation slightly by writing $m_{\lambda}$ for the symmetric polynomial in all the non-differentiated variables.
    \begin{align*}
        \dl_{x_1}^2\dl_{x_2} f &= m_1(\vec{x})m_1(\vec{y}) + m_{11}(\vec{y})\\
        \dl_{x_1}\dl_{x_2}\dl_{x_3} f &= (x_1 + x_2 + x_3)m_1(\vec{y}) + m_1(\vec{x})m_1(\vec{y}) + \frac{1}{2}m_2(\vec{y}) + m_{11}(\vec{y})\\
        \dl_{x_1}^2\dl_{y_1}f &= m_{11}(\vec{x}) + m_1(\vec{x})m_1(\vec{y})\\
        \dl_{x_1}\dl_{x_2}\dl_{y_1} &= \frac{1}{2}m_2(\vec{x}) + (x_1 + x_2)m_1(\vec{x}) + m_{11}(\vec{x})\\
            &\qquad + (x_1 + x_2)m_1(\vec{y}) + m_1(\vec{x})y_1 + m_1(\vec{x})m_1(\vec{y})\\
            &\qquad + y_1 m_1(\vec{y}) + \frac{1}{2}m_2(\vec{y}) + m_{11}(\vec{y})\\
        \dl_{x_1}\dl_{y_1}^2 f &= m_{11}(\vec{x}) + m_1(\vec{x})m_1(\vec{y})\\
        \dl_{x_1}\dl_{y_1}\dl_{y_2} f &= x_1m_1(\vec{x}) + \frac{1}{2}m_2(\vec{x}) + m_{11}(\vec{x}) + m_1(\vec{x})(y_1 + y_2) + m_1(\vec{x})m_1(\vec{y}).
    \end{align*}
    
    One can check that all of these derivatives have Lorentzian signature.
\end{example}

The fact that $\langle \prod (x_i + y_j), s_\lambda(y)\rangle$ is denormalized Lorentzian seems to be more a coincidence than a mathematically meaningful fact. A fundamental limitation of this inner product as an operator is that the output only depends on the part of $f$ that is homogeneous of degree $|\lambda|$ in the $y$ variables and degree $m+n-|\lambda|$ in the $x$ variables. We know that initial forms of Lorentzian polynomials are Lorentzian, so if we take only the terms of $f$ with maximal (resp. minimal) degree in $y$, then the result is Lorentzian, but in general there are no guarantees on how the intermediate degree parts of $f$ must behave. As such, any operator that only depends on one graded piece of $f$ should not be Lorentzianity-preserving, and as we showed above, restricting ourselves to symmetric polynomials does not fix this problem. Restricting to stable polynomials as our input is also likely insufficient to guarantee Lorentzianity of the output, as stability again only gives guarantees on the top- and bottom- degree parts.

\section{Conclusion and Future Directions}
This paper explores the intersection of Lorentzian polynomials and symmetric functions. We give some general theorems on their topology and structure, and we exploit this structure to give complete characterizations in small degree and to simplify some known proofs of Lorentzianity. While we show that many natural operations on symmetric functions do not interact nicely with the Lorentzian property, there are still several interesting directions to explore.

One future direction might be to explore the Lorentzianity of chromatic symmetric polynomials, following on the work in \cite{Matherne24chromaticsymmetric}, and in particular to try to characterize which graphs have Lorentzian chromatic symmetric functions. By analyzing the support conditions, we know that the graph must be claw-free. From \cite{Matherne24chromaticsymmetric}, we know that abelian Dyck paths give Lorentzian chromatic symmetric functions, but that not all non-abelian Dyck paths have this property.

From a slightly different direction, we know that the chromatic polynomial of any graph has log-concave coefficients~\cite[Theorem 9.9 (4)]{AdiprasitoHuhKatz2018}. Historically, the chromatic symmetric function has been regarded as the natural multivariate lift of the chromatic polynomial. Since not all chromatic symmetric functions are Lorentzian, it is natural to ask if there is another multivariate lift that does preserve the Lorentzian property.

It is also natural to study certain subtypes of Lorentzian symmetric polynomials, such as M-convex functions. As shown in \cite{BrandenHuh20Lorentzian}, M-convex functions give rise to a special, restrictive class of Lorentzian polynomials. In future work, we are looking to study the intersection of this class with symmetric polynomials. M-convexity imposes further restrictions on the coefficients of the polynomial, which should have interesting consequences when combined with the symmetry restrictions.

Last but not least, a sensible direction is to continue studying the obvious stratification of the projective space of Lorentzian symmetric polynomials in the same spirit as recent work of \citeauthor*{baker2025lorentzianpolynomialsmatroidstriangular} \cite{baker2025lorentzianpolynomialsmatroidstriangular}. Answering general questions like computing the dimension and other topological invariants of each piece would be a natural next step.

\bibliographystyle{plainnat}
\bibliography{ref}

\begin{thebibliography}{21}
\providecommand{\natexlab}[1]{#1}
\providecommand{\url}[1]{\texttt{#1}}
\expandafter\ifx\csname urlstyle\endcsname\relax
  \providecommand{\doi}[1]{doi: #1}\else
  \providecommand{\doi}{doi: \begingroup \urlstyle{rm}\Url}\fi

\bibitem[Adiprasito et~al.(2018)Adiprasito, Huh, and Katz]{AdiprasitoHuhKatz2018}
Karim Adiprasito, June Huh, and Eric Katz.
\newblock Hodge theory for combinatorial geometries.
\newblock \emph{Annals of Mathematics}, 188\penalty0 (2):\penalty0 381--452, 2018.
\newblock \doi{10.4007/annals.2018.188.2.1}.

\bibitem[Alexandersson and Jal(2025)]{alexanderssonjal2025rookmatroids}
Per Alexandersson and Aryaman Jal.
\newblock Rook matroids and log-concavity of $p$-eulerian polynomials, 2025.
\newblock URL \url{https://arxiv.org/abs/2410.00127}.

\bibitem[Anari et~al.(2024)Anari, Liu, Oveis~Gharan, and Vinzant]{anari2024masons3}
Nima Anari, Kuikui Liu, Shayan Oveis~Gharan, and Cynthia Vinzant.
\newblock Log-concave polynomials {III}: {M}ason's ultra-log-concavity conjecture for independent sets of matroids.
\newblock \emph{Proc. Amer. Math. Soc.}, 152\penalty0 (5):\penalty0 1969--1981, 2024.
\newblock ISSN 0002-9939,1088-6826.
\newblock \doi{10.1090/proc/16724}.
\newblock URL \url{https://doi.org/10.1090/proc/16724}.

\bibitem[Baker et~al.(2025)Baker, Huh, Kummer, and Lorscheid]{baker2025lorentzianpolynomialsmatroidstriangular}
Matthew Baker, June Huh, Mario Kummer, and Oliver Lorscheid.
\newblock Lorentzian polynomials and matroids over triangular hyperfields 1: Topological aspects, 2025.
\newblock URL \url{https://arxiv.org/abs/2508.02907}.

\bibitem[Bapat and Raghavan(1997)]{bapat1997nonnegativematrices}
R.~B. Bapat and T.~E.~S. Raghavan.
\newblock \emph{Nonnegative matrices and applications}, volume~64 of \emph{Encyclopedia of Mathematics and its Applications}.
\newblock Cambridge University Press, Cambridge, 1997.
\newblock ISBN 0-521-57167-7.
\newblock \doi{10.1017/CBO9780511529979}.
\newblock URL \url{https://doi.org/10.1017/CBO9780511529979}.

\bibitem[Br\"and\'en(2021)]{branden2021euclideanball}
Petter Br\"and\'en.
\newblock Spaces of {L}orentzian and real stable polynomials are {E}uclidean balls.
\newblock \emph{Forum Math. Sigma}, 9:\penalty0 Paper No. e73, 8, 2021.
\newblock ISSN 2050-5094.
\newblock \doi{10.1017/fms.2021.70}.
\newblock URL \url{https://doi.org/10.1017/fms.2021.70}.

\bibitem[Br\"and\'en and Huh(2020)]{BrandenHuh20Lorentzian}
Petter Br\"and\'en and June Huh.
\newblock Lorentzian polynomials.
\newblock \emph{Ann. of Math. (2)}, 192\penalty0 (3):\penalty0 821--891, 2020.
\newblock ISSN 0003-486X,1939-8980.
\newblock \doi{10.4007/annals.2020.192.3.4}.
\newblock URL \url{https://doi.org/10.4007/annals.2020.192.3.4}.

\bibitem[Br\"and\'en et~al.(2023)Br\"and\'en, Leake, and Pak]{branden2023lowerbounds}
Petter Br\"and\'en, Jonathan Leake, and Igor Pak.
\newblock Lower bounds for contingency tables via {L}orentzian polynomials.
\newblock \emph{Israel J. Math.}, 253\penalty0 (1):\penalty0 43--90, 2023.
\newblock ISSN 0021-2172,1565-8511.
\newblock \doi{10.1007/s11856-022-2364-9}.
\newblock URL \url{https://doi.org/10.1007/s11856-022-2364-9}.

\bibitem[Brändén and Huh(2019)]{branden2019hodgeriemannrelationspottsmodel}
Petter Brändén and June Huh.
\newblock Hodge-{R}iemann relations for {P}otts model partition functions, 2019.
\newblock URL \url{https://arxiv.org/abs/1811.01696}.

\bibitem[Chin(2024)]{chin2024realstabilitylogconcavity}
Tracy Chin.
\newblock Real stability and log concavity are co{NP}-hard, 2024.
\newblock URL \url{https://arxiv.org/abs/2405.00162}.

\bibitem[Fulton and Harris(1991)]{fulton1991reptheory}
William Fulton and Joe Harris.
\newblock \emph{Representation theory}, volume 129 of \emph{Graduate Texts in Mathematics}.
\newblock Springer-Verlag, New York, 1991.
\newblock ISBN 0-387-97527-6; 0-387-97495-4.
\newblock \doi{10.1007/978-1-4612-0979-9}.
\newblock URL \url{https://doi.org/10.1007/978-1-4612-0979-9}.
\newblock A first course, Readings in Mathematics.

\bibitem[Galashin et~al.(2018)Galashin, Karp, and Lam]{galashin2018contractiveflow}
Pavel Galashin, Steven~N. Karp, and Thomas Lam.
\newblock The totally nonnegative {G}rassmannian is a ball.
\newblock \emph{S\'em. Lothar. Combin.}, 80B:\penalty0 Art. 23, 12, 2018.
\newblock ISSN 1286-4889.

\bibitem[Hafner et~al.(2025)Hafner, Mészáros, and Vidinas]{hafner2025alexanderpolynomialspecialalternating}
Elena~S. Hafner, Karola Mészáros, and Alexander Vidinas.
\newblock On the alexander polynomial of special alternating links, 2025.
\newblock URL \url{https://arxiv.org/abs/2401.14927}.

\bibitem[Huh et~al.(2022)Huh, Matherne, M\'esz\'aros, and {St.\ Dizier}]{huh2022schurlorentzian}
June Huh, Jacob~P. Matherne, Karola M\'esz\'aros, and Avery {St.\ Dizier}.
\newblock Logarithmic concavity of {S}chur and related polynomials.
\newblock \emph{Trans. Amer. Math. Soc.}, 375\penalty0 (6):\penalty0 4411--4427, 2022.
\newblock ISSN 0002-9947,1088-6850.
\newblock \doi{10.1090/tran/8606}.
\newblock URL \url{https://doi.org/10.1090/tran/8606}.

\bibitem[Khare et~al.(2025)Khare, Matherne, and {St. Dizier}]{khare2025logconcavitycharactersparabolicverma}
Apoorva Khare, Jacob~P. Matherne, and Avery {St. Dizier}.
\newblock Log-concavity of characters of parabolic verma modules, and of restricted kostant partition functions, 2025.
\newblock URL \url{https://arxiv.org/abs/2504.01623}.

\bibitem[Matherne et~al.(2024)Matherne, Morales, and Selover]{Matherne24chromaticsymmetric}
Jacob~P. Matherne, Alejandro~H. Morales, and Jesse Selover.
\newblock The {N}ewton polytope and {L}orentzian property of chromatic symmetric functions.
\newblock \emph{Selecta Math. (N.S.)}, 30\penalty0 (3):\penalty0 Paper No. 42, 35, 2024.
\newblock ISSN 1022-1824,1420-9020.
\newblock \doi{10.1007/s00029-024-00928-4}.
\newblock URL \url{https://doi.org/10.1007/s00029-024-00928-4}.

\bibitem[Murota(2003)]{murota2003discreteconvexanalysis}
Kazuo Murota.
\newblock \emph{Discrete convex analysis}.
\newblock SIAM Monographs on Discrete Mathematics and Applications. Society for Industrial and Applied Mathematics (SIAM), Philadelphia, PA, 2003.
\newblock ISBN 0-89871-540-7.
\newblock \doi{10.1137/1.9780898718508}.
\newblock URL \url{https://doi.org/10.1137/1.9780898718508}.

\bibitem[Qin(2023)]{qinthesis}
Daniel Qin.
\newblock Symmetric {L}orentzian polynomials.
\newblock Master's thesis, KTH Royal Institute of Technology, 2023.

\bibitem[Rado(January 1952)]{RR52}
Richard Rado.
\newblock An inequality.
\newblock \emph{{J}ournal of the {L}ondon {M}athematical {S}ociety}, 27\penalty0 (1):\penalty0 1--6, January 1952.

\bibitem[Stanley(1999)]{stanley1999ec2}
Richard~P. Stanley.
\newblock \emph{Enumerative combinatorics. {V}ol. 2}, volume~62 of \emph{Cambridge Studies in Advanced Mathematics}.
\newblock Cambridge University Press, Cambridge, 1999.
\newblock ISBN 0-521-56069-1; 0-521-78987-7.
\newblock \doi{10.1017/CBO9780511609589}.
\newblock URL \url{https://doi.org/10.1017/CBO9780511609589}.
\newblock With a foreword by Gian-Carlo Rota and appendix 1 by Sergey Fomin.

\bibitem[White(1980)]{white1980monotonicity}
Dennis~E. White.
\newblock Monotonicity and unimodality of the pattern inventory.
\newblock \emph{Adv. in Math.}, 38\penalty0 (1):\penalty0 101--108, 1980.
\newblock ISSN 0001-8708.
\newblock \doi{10.1016/0001-8708(80)90059-6}.
\newblock URL \url{https://doi.org/10.1016/0001-8708(80)90059-6}.

\end{thebibliography}
\end{document}